\documentclass[10pt]{article}

\usepackage{amsmath}
\usepackage{amsthm}
\usepackage{amssymb}
\usepackage{amsfonts}
\usepackage{graphicx}
\usepackage{verbatim}
\usepackage{mathrsfs}
\usepackage{subfigure}
\usepackage{fancyhdr}
\usepackage{latexsym}
\usepackage{graphicx}
\usepackage{dsfont}
\usepackage{epsf} 

\theoremstyle{plain}
\newtheorem{theorem}{Theorem}[section]
\newtheorem{proposition}[theorem]{Proposition}
\newtheorem{lemma}[theorem]{Lemma}
\newtheorem{corollary}[theorem]{Corollary}

\newtheorem{remark}[theorem]{Remark}

\theoremstyle{definition}

\newcommand{\R}{\mathbb{R}}
\newcommand{\Z}{\mathbb{Z}}
\newcommand{\cL}{\mathcal{L}}
\newcommand{\bx}{\bar{x}}
\newcommand{\hx}{\hat{x}}
\newcommand{\tx}{\tilde{x}}
\newcommand{\ty}{\tilde{y}}
\newcommand{\be}{{\bf e_1}}
\newcommand{\bydef}{\stackrel{\mbox{\tiny\textnormal{\raisebox{0ex}[0ex][0ex]{def}}}}{=}}

\def\N{\mathbb{N}}

\title{Rigorous numerics for nonlinear operators with tridiagonal dominant linear part}

\author{Maxime Breden \thanks{CMLA, ENS Cachan \& CNRS, 61 avenue du Pr\'esident Wilson, 94230 Cachan, France. {\tt mbreden@ens-cachan.fr}}
\and Laurent Desvillettes \thanks{CMLA, ENS Cachan \& CNRS, 61 avenue du Pr\'esident Wilson, 94230 Cachan, France. {\tt desville@cmla.ens-cachan.fr}}
\and Jean-Philippe Lessard\thanks {D\'epartement de Math\'ematiques et de Statistique, Universit\'e Laval, 1045 avenue de la M\'edecine, Qu\'ebec, QC, G1V0A6, Canada. {\tt jean-philippe.lessard@mat.ulaval.ca}}}

\begin{document}

\maketitle

\begin{abstract}
We present a method designed for computing solutions of infinite dimensional nonlinear operators $f(x)=0$ with a tridiagonal dominant linear part. We recast the operator equation into an equivalent Newton-like equation $x=T(x)=x-Af(x)$, where $A$ is an approximate inverse
 of the derivative $Df(\bx)$ at an approximate solution $\bx$. We present rigorous computer-assisted calculations showing that $T$ is a contraction
 near $\bx$, thus yielding the existence  of a solution. Since $Df(\bx)$ does not have an asymptotically diagonal dominant structure, the computation of $A$ is not straightforward. This paper provides ideas for computing $A$, and proposes a new rigorous  method for
proving existence of solutions of nonlinear operators with tridiagonal dominant linear part.
\end{abstract}

\begin{center}
{\bf \small Keywords} \\ \vspace{.05cm}
{ \small Tridiagonal operator $\cdot$ Contraction mapping $\cdot$ Rigorous numerics $\cdot$ Fourier series }
\end{center}

\begin{center}
{\bf \small Mathematics Subject Classification (2010)} \\ \vspace{.05cm}
{ \small 47H10 $\cdot$ 97N20 $\cdot$ 42A10 $\cdot$ 65L10 $\cdot$ 34B08}
\end{center}

\section{Introduction}
\label{sec:Intro}

Tridiagonal operators naturally arise in the theory of orthogonal polynomials, ordinary differential equations (ODEs), continued fractions, numerical analysis of partial differential equations (PDEs), integrable systems, quantum mechanics and solid state physics. Some differential operators can be represented by infinite tridiagonal matrices acting in sequence spaces, as it is the case for instance for differentiation in frequency space of the Hermite functions. Other examples come from the study of ODEs like the Mathieu equation, the spheroidal wave equation, the Whittaker-Hill equation and the Lam\'e equation. 

While  many well-developed methods and efficient algorithms already exist in the literature for solving linear tridiagonal matrix equations and computing their inverses, our own
 method has a different flavour. We aim at developing a computational method in order to prove, in a mathematically rigorous and constructive sense, existence of solutions to
 infinite dimensional nonlinear equations of the form
\begin{equation} \label{eq:general_f=0}
f(x)=\cL(x)+N(x)=0,
\end{equation}
where $\cL$ is a tridiagonal linear operator and $N$ is a nonlinear operator. The domain of the operator $f$ is the space of algebraically decaying sequences 
\begin{equation} \label{eq:Omega^s}
\Omega^s \bydef \left\{ x = (x_k)_{k \ge 0} : \|x\|_s \bydef \sup_{k \ge 0} \{|x_k| \omega_k^s\} < \infty \right\},
\end{equation}
where
\[
\omega_k^s \bydef  \left\{
\begin{array}{ll}
1, & k=0, \\
k^s,&  k\geq 1. 
\end{array}
\right.
\]

The assumptions on the linear and nonlinear parts of \eqref{eq:general_f=0} are that $\cL:\Omega^s \rightarrow \Omega^{s-s_L}$ and $N:\Omega^s \rightarrow \Omega^{s-s_N}$, 
for some $s_L>s_N$. Intuitively, this means that the linear part {\em dominates} the nonlinear part. Since $\Omega^{s_1} \subset \Omega^{s_2}$ for $s_1>s_2$, one can see that $f$ maps $\Omega^s$ into $\Omega^{s-s_L}$.

General nonlinear operator equations of the form $f(x)=0$ defined on the Banach space $\Omega^s$ arise in the
 study of bounded solutions of finite and infinite dimensional dynamical systems. For instance, $x=(x_k)_{k \ge 0}$ may be the infinite sequence of Fourier coefficients of a periodic solution of an ODE, a periodic solution of a delay differential equation (DDE) or an equilibrium solution of a PDE with Dirichlet, periodic or Neumann boundary conditions. The unknown $x$ may also be the infinite sequence of Chebyshev coefficients of a solution of a boundary value problem (BVP), the Hermite coefficients of a solution of an ODE defined on an unbounded domain, or the Taylor
 coefficients of the solution of a Cauchy problem. In the case when the differential equation is smooth, the decay rate of the coefficients of $x$ will be algebraic or even exponential \cite{MR1874071}. In the present paper, we chose to solve \eqref{eq:general_f=0} in the weighed $\ell^\infty$ Banach space $\Omega^s$ which corresponds to $C^k$ solutions. In order to exploit the analyticity of the solutions, we could follow the idea of \cite{HLM} and solve \eqref{eq:general_f=0} in weighed $\ell^1$ Banach spaces. This choice of space is not considered in the present paper.

Recently, several attempts to solve $f(x)=0$ in $\Omega^s$ have been successful. They belong to a field now called {\em rigorous numerics}. This field aims at constructing algorithms that provide approximate solutions to a given problem, together with precise bounds implying the existence
of an exact solution in the mathematically rigorous sense. Equilibria of PDEs \cite{MR1838755,MR2126387,MR2718657}, periodic solutions of DDEs \cite{MR2871794}, fixed points of infinite dimensional maps \cite{MR2067140} and periodic solutions of ODEs \cite{MR2166710,MR3032858} have been computed using such methods. 

One popular idea in rigorous numerics is to recast the problem $f(x)=0$ as a problem of fixed point of a Newton-like equation of the form $T(x)=x-Af(x)$, where $A$ is an approximate inverse of $Df(\bx)$, and $\bx$ is a numerical approximation obtained by computing a finite dimensional projection of $f$. In \cite{MR1838755,MR2126387,MR2871794, MR2067140,MR3032858,MR2718657}, the nonlinear equations under study have asymptotically diagonal or block-diagonal dominant linear part,  which helps a lot in the computation of approximate inverses.
 In contrast, the present work considers problems with tridiagonal dominant linear part. To the best of our knowledge, this is the first attempt to compute rigorously solutions of such problems. While our proposed approach is designed for a specific class of operators (see assumptions \eqref{eq:asump1_tridiag} and \eqref{eq:asump2_tridiag}), we believe that it can be seen as a first step toward rigorously solving more complicated nonlinear operators with tridiagonal dominant linear part.

The paper is organized as follows. In Section~\ref{sec:tridiagonal}, we present a method enabling to compute (with the help of the computer)
 pseudo-inverses of tridiagonal operators of a certain class. In Section~\ref{sec:rigorous_computational method}, we recast the problem $f(x)=0$ as a fixed point problem $T(x)=x-Af(x)$,
 where $A$ is a pseudo-inverse, and we present the rigorous computational method to prove existence of fixed points of $T$. In Section~\ref{sec:example}, we present an application and finally,
 in Section~\ref{sec:conclusion}, we conclude by presenting some interesting future directions. 

\section{Computing pseudo-inverses of tridiagonal operators}
\label{sec:tridiagonal}

This Section is devoted to the construction of a pseudo-inverse of a linear operator with tridiagonal tail (see \eqref{eq:opertor_A_dag}). We begin this Section by specifying the assumptions that
 we make on the growth of the tridiagonal terms. Then we use an LU-decomposition to formally obtain a formula for the pseudo-inverse. Finally,
 we check that the (formally defined) pseudo-inverse has good mapping properties
%does in fact map between the right spaces
 (see Proposition~\ref{prop:A_tri_domain_image}).
\bigskip

Given three sequences $(\lambda_k)_{k \ge 0}$, $(\mu_k)_{k \ge 0}$, $(\beta_k)_{k \ge 0}$ and $x \in \Omega^s$,
we define the tridiagonal linear operator (acting on $x$) $\cL(x)=(\cL_k(x))_{k \ge 0}$ of \eqref{eq:general_f=0} by

\begin{equation} \label{eq:L_k_tridiagonal}
\cL_k(x) =\lambda_{k} x_{k-1} + \mu_k x_k + \beta_k x_{k+1}, ~~ k \ge 1,
\end{equation}
and $\cL_0(x) =  \mu_0 x_0 + \beta_0 x_{1}$. Assume that there exist real numbers $s_L>0$, $0<C_1\leq C_2$ and an integer $k_0$ such that 
\begin{equation} \label{eq:asump1_tridiag}
\forall ~ k\geq 0,\quad \left| \frac{\lambda_k}{\omega_k^{s_L}}\right|,\left| \frac{\mu_k}{\omega_k^{s_L}}\right|,\left| \frac{\beta_k}{\omega_k^{s_L}}\right|\leq C_2 \quad \text{and} \quad  \forall ~ k\geq k_0,\quad C_1\leq \left| \frac{\mu_k}{\omega_k^{s_L}}\right|.
\end{equation}
Assume further the existence of $\displaystyle \delta \in \left(0,\frac{1}{2}\right)$ and $k_0 \ge 0$ such that
\begin{equation} \label{eq:asump2_tridiag}
\forall  ~k\geq k_0,\quad \left| \frac{\lambda_k}{\mu_k} \right|, \left| \frac{\beta_k}{\mu_k} \right| \le \delta.
\end{equation}
Then, under assumptions \eqref{eq:asump1_tridiag} and \eqref{eq:asump2_tridiag}, $\cL$ defined by \eqref{eq:L_k_tridiagonal} is a tridiagonal operator which maps $\Omega^s$ into $\Omega^{s-s_L}$.  Indeed, if $x \in \Omega^s$, then
\begin{eqnarray*}
\|\cL(x)\|_{s-s_L} &=& \sup_{k \ge 0} \{| \cL_k(x)  | \omega_k^{s-s_L}\} \\
&\le& C_2 \left( \sup_{k \geq  1} \{| x_{k-1} | \omega_k^s \} + \sup_{k \geq  0} \{| x_k | \omega_k^s \} +  \sup_{k \geq  0} \{|x_{k+1}| \omega_k^s\} \right) 
< \infty.
\end{eqnarray*}

From now on, assume for the sake of simplicity that $s_N=0$, that is the nonlinear part $N$ of \eqref{eq:general_f=0} maps $\Omega^s$ into $\Omega^s$. 
Since $\Omega^s$ is an algebra under discrete convolutions when $s>1$ (e.g. see \cite{MR2718657,MR3125637}), then any $N$ which is a combination of such convolutions maps $\Omega^s$ into $\Omega^s$. Assume that using a finite dimensional projection $f^{(m)}:\R^m \rightarrow \R^m$ of \eqref{eq:general_f=0}, we computed a numerical approximation $\bx$ such that $f^{(m)}(\bx) \approx 0$. We identify $\bx \in \R^m$ and $\bx =(\bx,0,0,0,0,\dots) \in \Omega^s$.
We then try to construct a ball 
\[
B_{\bx}(r) = \bx + B_{0}(r) = \bx +  \left\{ x \in \Omega^s : \| x \|_s \le r \right\} = \left\{ x \in \Omega^s : \| x - \bx\|_s \le r \right\}
\]
centered at $\bx$ and containing a unique solution of  \eqref{eq:general_f=0}, by showing that a specific
 Newton-like operator $T(x)=x-Af(x)$ is a contraction on $B_{\bx}(r)$. This requires the construction of
an approximate inverse $A$ of $Df(\bx) = \cL (\bx) + DN(\bx)$. In order to do so, the structures of $\cL(\bx)$ and $DN(\bx)$ need to be understood. From \eqref{eq:L_k_tridiagonal} and \eqref{eq:asump1_tridiag},  $\cL(\bx)$ is a tridiagonal operator with entries growing to infinity at the rate $k^{s_L}$. Moreover, since $DN(\bx)$ maps $\Omega^s$ into $\Omega^s$, it is a bounded linear operator. As mentioned above, the expectation is that the coefficients of $\bx$ decay fast to zero. This implies that a reasonable approximation $A^\dagger$ of $Df(\bx)$ is given by

\begin{equation}
\label{eq:opertor_A_dag}
A^\dagger \bydef  \left( \begin{array}{cccccc}
&   &   &   &   &   \\
& D &   &   & 0 &   \\
&   &   & \beta_{m-1} &   &  \\
&   & \lambda_{m} & \mu_m & \beta_{m} &  \\ 
&  0  &             &  \lambda_{m+1} & \mu_{m+1} & \beta_{m+1}
\end{array} \right),
\end{equation}
with $D \bydef Df^{(m)}(\bx)$ for $m$ large enough. We wish to find the inverse of $A^{\dag}$ in terms of $D$, $(\beta_k)_{k \ge m-1}$, $(\mu_k)_{k\ge m}$ and $(\lambda_k)_{k\ge m}$. We assume therefore that 
\begin{equation} \label{eq:Adagx=y}
A^\dagger x = y, 
\end{equation}
 where $x$ and $y$ are the infinite vectors
$$ 
x = \left( \begin{array}{c} x_0 \\ x_1 \\ . \\ . \\ \\ \end{array} \right), \qquad y= \left( \begin{array}{c} y_0\\ y_1 \\ . \\ . \\ \\ \end{array} \right). 
$$  
The infinite part of \eqref{eq:Adagx=y} writes
\begin{equation} \label{eq:tridiagonal_tail}
\left( \begin{array}{ccccc}
 \mu_m    &  \beta_m   & 0           &          0  & ...     \\
 \lambda_{m+1}  &  \mu_{m+1} & \beta_{m+1} &          0  & ...   \\
   0      &  \lambda_{m+2} & \mu_{m+2} & \beta_{m+2} & ...  \\
  . & .  & . &  . &  ...  
   \end{array} \right) \,\, \left( \begin{array}{c} x_m \\ x_{m+1} \\ . \\ . \\ \\ \end{array} \right)
= \left( \begin{array}{c} y_m - \lambda_m\, x_{m-1} \\ y_{m+1} \\ . \\ . \\ \\ \end{array} \right).
\end{equation}

We introduce the notations of the book of P.G. Ciarlet (see Theorem~4.3-2 on page 142 in \cite{MR1015713}):
 $$ a_2= \lambda_{m+1}, \quad a_3 = \lambda_{m+2},..., \qquad 
 b_1=\mu_m, \quad b_2= \mu_{m+1},..., \qquad
 c_1=\beta_m, \quad c_2=\beta_{m+1},..., $$
 and $(\delta_n)_{n\in \N}$ defined by the induction formula
$$\delta_0 = 1,~~ \delta_1=b_1,~~{\rm and}~~ \delta_n = b_n \, \delta_{n-1} - a_n\, c_{n-1}\, \delta_{n-2}, ~~~ {\rm for}~~ n \ge 2.$$
Note that only the $\delta_n$ are really useful.

Let us define the tridiagonal operator $T$ by 
\begin{equation} \label{eq:Ciarlet_T}
T \bydef  \left( \begin{array}{ccccc}
b_1    &  c_1   & 0           &          0  & ...     \\
a_2  &  b_2 & c_2 &          0  & ...   \\
   0      &  a_3 & b_3 & c_3 & ...  \\
  . & .  & . &  . &  ...  
\end{array} \right).
\end{equation}
For any infinite vector $x=(x_0,\ldots,x_k,\ldots)^T$, we introduce the notation
\begin{equation*}
x_F \bydef  (x_0,\ldots,x_{m-1})^T \quad\text{and}\quad x_I \bydef  (x_m,\ldots,x_{m+k},\ldots)^T.
\end{equation*}
Using the notation $\be=(1,0,0,0,0,\cdots)^T$, the system \eqref{eq:tridiagonal_tail} becomes 
$$ 
T x_I = y_I -  \lambda_{m}\, x_{m-1} \be.
$$

From Theorem~4.3-2 in \cite{MR1015713}, we compute an $LU$-decomposition of the tridiagonal operator defined in \eqref{eq:Ciarlet_T} as $T=L_I U_I$, where
\begin{equation} \label{eq:LandU}
L_I \bydef 
\left( \begin{array}{cccc}
 1    &  0   & 0           &     ...     \\
 a_2\, \frac{\delta_0}{\delta_1} &  1 &          0  & ...   \\
   0      &  a_3\, \frac{\delta_1}{\delta_2} & 1 &  ...  \\
  . & .  & . &   ...  
   \end{array} \right) ~~ {\rm and} ~~
U_I \bydef   \left( \begin{array}{cccc}
 \frac{\delta_1}{\delta_0}    &  c_1                      & 0           &     ...     \\
 0                            & \frac{\delta_2}{\delta_1} &  c_2        &     ...   \\
  0                            & 0 &  \frac{\delta_3}{\delta_2}        &     ...   \\
  . & .  & . &   ...  
   \end{array} \right).
\end{equation}
Hence, the system \eqref{eq:tridiagonal_tail} becomes $L_I z_I = y_I  -  \lambda_m \, x_{m-1} \be$ combined with
$U_I x_I = z_I$, that is
\begin{equation} \label{eq:L-system}
\left( \begin{array}{cccc}
 1    &  0   & 0           &     ...     \\
 a_2\, \frac{\delta_0}{\delta_1} &  1 &          0  & ...   \\
   0      &  a_3\, \frac{\delta_1}{\delta_2} & 1 &  ...  \\
  . & .  & . &   ...  
   \end{array} \right) \,\, \left( \begin{array}{c} z_m \\ z_{m+1} \\ . \\ . \\ \\ \end{array} \right)
= \left( \begin{array}{c} y_m - \lambda_m\, x_{m-1} \\ y_{m+1} \\ . \\ . \\ \\ \end{array} \right),
\end{equation}
combined with
\begin{equation} \label{eq:U-system}
\left( \begin{array}{cccc}
 \frac{\delta_1}{\delta_0}    &  c_1                      & 0           &     ...     \\
 0                            & \frac{\delta_2}{\delta_1} &  c_2        &     ...   \\
  0                            & 0 &  \frac{\delta_3}{\delta_2}        &     ...   \\
  . & .  & . &   ...  
   \end{array} \right) \,\, \left( \begin{array}{c} x_m \\ x_{m+1} \\ .  \\ \\ \end{array} \right)
= \left( \begin{array}{c} z_m  \\ z_{m+1} \\ . \\ \\ \end{array} \right).
\end{equation} 
Both infinite systems \eqref{eq:L-system} and \eqref{eq:U-system} can be explicitly solved.

System \eqref{eq:L-system} leads to 
\begin{equation*}
z_{m} = y_{m}-\lambda_m x_{m-1},
\end{equation*}
and for any $k\ge 1$
$$ z_{m+k} = y_{m+k} + \sum_{l=1}^k (-1)^l \, a_{k-l+2}\,..\,a_{k+1}\, \frac{\delta_{k-l}}{\delta_k}\,
y_{m+k-l} + (-1)^{k+1}\, a_2\,..\,a_{k+1} \,\frac{\delta_{0}}{\delta_k}\,\lambda_m x_{m-1},$$
which we rewrite with infinite matrix/vectors notations as
\begin{equation}\label{1i}
z_I = {L_I}^{-1} [y_I  -  \lambda_m\, x_{m-1} \be ] = {L_I}^{-1} y_I - \lambda_m x_{m-1} \, v_I, 
\end{equation}
where
$$ 
z_I = \left( \begin{array}{c} z_m \\ z_{m+1} \\ z_{m+2}\\  \\ \vdots \\ \\ \end{array} \right), ~~
y_I = \left( \begin{array}{c} y_m\\ y_{m+1} \\ y_{m+2}\\  \\ \vdots \\ \\ \end{array} \right), ~~ 
v_I \bydef  {L_I}^{-1} \be = \,\left( \begin{array}{c} 1\\ -a_2\, \frac{\delta_{0}}{\delta_1}\\ a_3\,a_2\,
\frac{\delta_{0}}{\delta_2}  \\  -a_4\,a_3\,a_2\, \frac{\delta_{0}}{\delta_3}\\ \vdots \\ \\ \end{array} \right). 
$$  

The second system \eqref{eq:U-system} leads to the infinite sum (for any $k\ge 0$)
\begin{equation*}
x_{m+k} = \frac{\delta_{k}}{\delta_{k+1}} z_{m+k} + \sum_{l=1}^{\infty} (-1)^l\, \frac{\delta_k}{\delta_{k+l+1}} \, c_{k+1}\,..\,c_{k+l}\, z_{m+k+l},
\end{equation*}
which we also rewrite with infinite matrix/vector notations as 
\begin{equation}\label{2i}
x_I = {U_I}^{-1} z_I.
\end{equation}
Coupling (\ref{1i}) and (\ref{2i}), we end up with 
\begin{equation}\label{3i}
x_I =  {U_I}^{-1} z_I =  {U_I}^{-1} [ {L_I}^{-1} y_I - \lambda_m x_{m-1} \, v_I] = {U_I}^{-1} {L_I}^{-1} y_I - \lambda_m x_{m-1} w_I,
\end{equation}
where $w_I \bydef  U_I^{-1}v_I$.
Denoting $\left( {U_I}^{-1} {L_I}^{-1}\right)_{r_0}$ the first row of the infinite matrix ${U_I}^{-1} {L_I}^{-1}$ and $\left( w_I \right)_{0}$ the first element of $w_I$, we can rewrite the first line of (\ref{3i}) as 

\begin{equation}\label{4i}
x_m =  \left( {U_I}^{-1} {L_I}^{-1} \right)_{r_0} y_I - \lambda_m x_{m-1} \left( w_I \right)_0. 
\end{equation}

We now investigate the finite part of the linear system \eqref{eq:Adagx=y}, which is given by
$$ D\, \left( \begin{array}{c} x_0 \\ x_{1} \\ . \\ . \\ x_{m-2} \\ x_{m-1}\\ \end{array} \right) 
+ \left( \begin{array}{c} 0\\ 0 \\ . \\ . \\ 0 \\ \beta_{m-1}\, x_m \\ \end{array} \right)=
\left( \begin{array}{c} y_0 \\ y_{1} \\ . \\ . \\ y_{m-2} \\ y_{m-1} \\ \end{array} \right),
$$

or, according to (\ref{4i}),

$$ D\, \left( \begin{array}{c} x_0 \\ x_{1} \\ . \\ . \\ x_{m-2} \\ x_{m-1}\\ \end{array} \right)  
+ \beta_{m-1}\, 
\left( \begin{array}{c} 0\\ 0 \\ . \\ . \\ 0 \\ \left( {U_I}^{-1} {L_I}^{-1} \right)_{r_0} y_I - \lambda_m x_{m-1} \left( w_I \right)_0  \\ \end{array} \right)=
\left( \begin{array}{c} y_0 \\ y_{1} \\ . \\ . \\ y_{m-2} \\ y_{m-1} \\ \end{array} \right).$$

Letting
$$ K \bydef D - \beta_{m-1}\lambda_m \,  \left( \begin{array}{ccccc}
 0    &  0   & ...   & 0   & 0      \\
 0    &  0   & ...   & 0   & 0      \\
 \vdots    &  \vdots   & \ddots   & \vdots   & \vdots      \\
 0    &  0   & ...   & 0   & 0      \\
 0    &  0   & ...   & 0   & \left( w_I \right)_{0}
\end{array} \right), $$
we consider its inverse $K^{-1}$. We denote the last column of $K^{-1}$ by $(K^{-1})_{c_{m-1}}$, its last row by $(K^{-1})_{r_{m-1}}$, and its last (``south-east") element by $(K^{-1})_{m-1,m-1}$. 
Then we obtain 
\begin{eqnarray}
x_F &=& K^{-1} y_F  - \beta_{m-1} \left\{ \left( {U_I}^{-1} {L_I}^{-1} \right)_{r_0} y_I  \right\} (K^{-1})_{c_{m-1}} \nonumber
\\
&=&  K^{-1} y_F  - \beta_{m-1}\, \bigg( \bigg\{ (K^{-1})_{c_{m-1}} \bigg\} \otimes \bigg\{ \left( U_I^{-1} {L_I}^{-1} \right)_{r_0}
\bigg\} \bigg) y_I ,  \label{1f}
\end{eqnarray}
using the tensor product notation. The last line of this identity reads 

\begin{equation}\label{2f}
 x_{m-1} = (K^{-1})_{r_{m-1}}\, y_F - \beta_{m-1}\,   
\left\{ \left( {U_I}^{-1} {L_I}^{-1} \right)_{r_0} y_I  \right\} (K^{-1})_{m-1,m-1}.
\end{equation}
Coming back to (\ref{3i}) and using \eqref{2f}, we see that

\begin{eqnarray}
\nonumber
x_I &=& {U_I}^{-1} {L_I}^{-1} y_I - \lambda_m x_{m-1} w_I \\
\nonumber
&=& {U_I}^{-1} {L_I}^{-1} y_I  \\
\nonumber
&& - \lambda_m \bigg[  
(K^{-1})_{r_{m-1}}\, y_F - \beta_{m-1} \left\{ \left( {U_I}^{-1} {L_I}^{-1} \right)_{r_0} y_I  \right\} (K^{-1})_{m-1,m-1}
\bigg]w_I \\
\nonumber
&=& {U_I}^{-1} {L_I}^{-1} y_I  -\lambda_m \, w_I\bigg\{ 
(K^{-1})_{r_{m-1}} \, y_F \bigg\} \\
\nonumber
&& + \beta_{m-1}\lambda_m \, (K^{-1})_{m-1,m-1}\, w_I  \bigg\{ \left( {U_I}^{-1} {L_I}^{-1} \right)_{r_0} y_I \bigg\} \\
&=& 
-\lambda_m\bigg( \bigg\{  w_I  \bigg\} \otimes \, \bigg\{  (K^{-1})_{r_{m-1}} \bigg\} \bigg) \, y_F
\label{5i}
\\ 
&&
+  \bigg( {U_I}^{-1} {L_I}^{-1} + \beta_{m-1}\lambda_m\, (K^{-1})_{m-1,m-1}\,    
\bigg\{   w_I  \bigg\} \otimes \, \bigg\{ \left(  {U_I}^{-1} {L_I}^{-1} \right)_{r_0} \bigg\}  \, \bigg) \, y_I.
\nonumber
\end{eqnarray}

Putting together (\ref{1f}) and (\ref{5i}), we end up with

$$ (A^\dagger)^{-1} = \left( \begin{array}{cc} 
 K^{-1} & 
 - \beta_{m-1}\, \bigg( \bigg\{ (K^{-1})_{c_{m-1}} \bigg\} \otimes \bigg\{ \left( U_I^{-1} {L_I}^{-1} \right)_{r_0}
\bigg\} \bigg)
  \\
 -\lambda_m \bigg\{ w_I \bigg\} \otimes \, \bigg\{  (K^{-1})_{r_{m-1}} \bigg\} &
 {U_I}^{-1} {L_I}^{-1} + \tilde \Lambda 
 \end{array} \right),$$
where
$$ \tilde  \Lambda \bydef \beta_{m-1}\lambda_m \, (K^{-1})_{m-1,m-1}\,    
\bigg\{   w_I  \bigg\} \otimes \, \bigg\{ \left(  {U_I}^{-1} {L_I}^{-1} \right)_{r_0} \bigg\}
%\bigg\{ [U_I]\,[V_I]\bigg\} \otimes \, \bigg\{ ([U_I]\,[L_I])_{r_0} \bigg\} 
. $$

In order to get an approximate (pseudo) inverse of $A^\dagger$, we would like to get a numerical approximation of $K^{-1}$.
 However the definition of $K$ involves $\left(w_I\right)_0$,
 which cannot be explicitly computed. By definition, $w_I=U_I^{-1}L_I^{-1} \be$, so using again the computations made in this Section,
 we get
\begin{align*}
\left(w_I\right)_0 & = \left(U_I^{-1} v_I\right)_0 \\
& = \frac{\delta_0}{\delta_1}v_m + \sum_{l=1}^{\infty}(-1)^l\frac{\delta_0}{\delta_{l+1}}c_1\ldots c_l\, v_{m+l} \\
& = \frac{\delta_0}{\delta_1} + \sum_{l=1}^{\infty}\frac{\delta_0^2}{\delta_l \delta_{l+1}}c_1\ldots c_l\, a_2\ldots a_{l+1}.
\end{align*}
Given a computational parameter $L$, we define
\begin{equation}\label{eq:tilde_w}
\tilde w \bydef \frac{\delta_0}{\delta_1} + \sum_{l=1}^{L-1}\frac{\delta_0^2}{\delta_l \delta_{l+1}}c_1\ldots c_l\, a_2\ldots a_{l+1},
\end{equation}
and
\begin{equation*}
\tilde K  \bydef D - \beta_{m-1}\lambda_m \,  \left( \begin{array}{ccccc}
 0    &  0   & ...   & 0   & 0      \\
 0    &  0   & ...   & 0   & 0      \\
 \vdots    &  \vdots   & \ddots   & \vdots   & \vdots      \\
 0    &  0   & ...   & 0   & 0      \\
 0    &  0   & ...   & 0   & \tilde w
\end{array} \right).
\end{equation*}

We now can consider $A_m$ a numerically computed inverse of $\tilde K$ and then define the approximate (pseudo) inverse of $A^\dagger$ as
\begin{equation} \label{eq:operator_A}
A  \bydef \left( \begin{array}{cc} 
A_m & 
 - \beta_{m-1}\, \bigg( \bigg\{ (A_m)_{c_{m-1}} \bigg\} \otimes \bigg\{ \left( U_I^{-1} {L_I}^{-1} \right)_{r_0}
\bigg\} \bigg)
  \\
 -\lambda_m \bigg\{ w_I \bigg\} \otimes \, \bigg\{  (A_m)_{r_{m-1}} \bigg\} &
 {U_I}^{-1} {L_I}^{-1} + \Lambda 
 \end{array} \right),
 \end{equation}
where
$$ \Lambda  \bydef \beta_{m-1}\lambda_m \, (A_m)_{m-1,m-1}\,    
\bigg\{   w_I  \bigg\} \otimes \, \bigg\{ \left(  {U_I}^{-1} {L_I}^{-1} \right)_{r_0} \bigg\}
. $$

%La matrice (infinie) obtenue est l'inverse exacte de la matrice (infinie) de d\'epart. Reste \`a d\'ecider par quelle approximation
%la remplacer pour la m\'ethode de Newton (approch\'ee). Le plus simple semble \^etre de remplacer la ``partie finie'' $K^{-1}$ par une approximation utilisable au niveau informatique, et de garder le reste au niveau ``abstrait'' en esp\'erant que les estimations
%fonctionnent et que tous les morceaux ``infinis'' soient suffisamment explicites....

\begin{lemma} \label{lem:U_I_inv}
Assume that $m\geq k_0$ and $\delta<\frac{1}{2}$. Then $U_I^{-1}$ maps $\Omega^s$ into $\Omega^{s+s_L}$.
\end{lemma}

\begin{proof}
Let $z_I\in \Omega^s$ and $x_I=U_I^{-1}z_I$. Using (\ref{2i}) and the formula above, we get
\begin{align}
\label{proof:U_I_inv_1}
\left\vert x_{m+k}\right\vert &\leq \frac{\left\vert \delta_k\right\vert}{\left\vert \delta_{k+1}\right\vert}\left\vert z_{m+k}\right\vert + \sum_{l=1}^{\infty} \frac{\left\vert \delta_k\right\vert}{\left\vert \delta_{k+l+1}\right\vert} \, \left\vert c_{k+1}\right\vert\,..\,\left\vert c_{k+l}\right\vert\, \left\vert z_{m+k+l}\right\vert   \nonumber \\ 
&\leq  \frac{\left\vert \delta_k\right\vert}{\left\vert \delta_{k+1}\right\vert}\left\vert z_{m+k}\right\vert + \sum_{l=1}^{\infty} \delta^l \frac{\left\vert \delta_k\right\vert}{\left\vert \delta_{k+l+1}\right\vert} \, \left\vert b_{k+1}\right\vert\,..\,\left\vert b_{k+l}\right\vert\, \left\vert z_{m+k+l}\right\vert.
\end{align}
Now remember that for all $k\geq 2$, $\delta_k = b_k \, \delta_{k-1} - a_k\, c_{k-1}\delta_{k-2}$, so
\begin{align*}
\frac{\left\vert \delta_k\right\vert}{\left\vert \delta_{k-1}\right\vert\left\vert b_k\right\vert} &\geq 1-\frac{\left\vert a_k\right\vert\left\vert c_{k-1}\right\vert\left\vert \delta_{k-2}\right\vert}{\left\vert b_k\right\vert\left\vert \delta_{k-1}\right\vert}\\
&\geq 1-\frac{\delta^2\left\vert b_{k-1}\right\vert \left\vert \delta_{k-2}\right\vert}{\left\vert \delta_{k-1}\right\vert}.
\end{align*}
We introduce $\displaystyle{u_k  \bydef \frac{\left\vert \delta_k\right\vert}{\left\vert \delta_{k-1}\right\vert\left\vert b_k\right\vert}}$ which then satisfies 
\begin{equation*}
\left\{
\begin{aligned}
&u_1=1,\\
&u_k\geq 1-\frac{\delta^2}{u_{k-1}},\quad \forall~ k\geq 2.
\end{aligned}
\right.
\end{equation*}
The study of the inductive sequence defined as above, but with $\ge$ replaced by $=$, yields that for any $k$, $\gamma\leq u_k \leq 1$, where $\gamma  \bydef \frac{1}{2}+\sqrt{\frac{1}{4}-\delta^2}$ is the largest root of $x=1-\frac{\delta^2}{x}$ (see Figure~\ref{suite_rec}).

\begin{figure} [htbp]
\begin{center}
\subfigure[$\delta=0.2$]{\includegraphics[width=68mm]{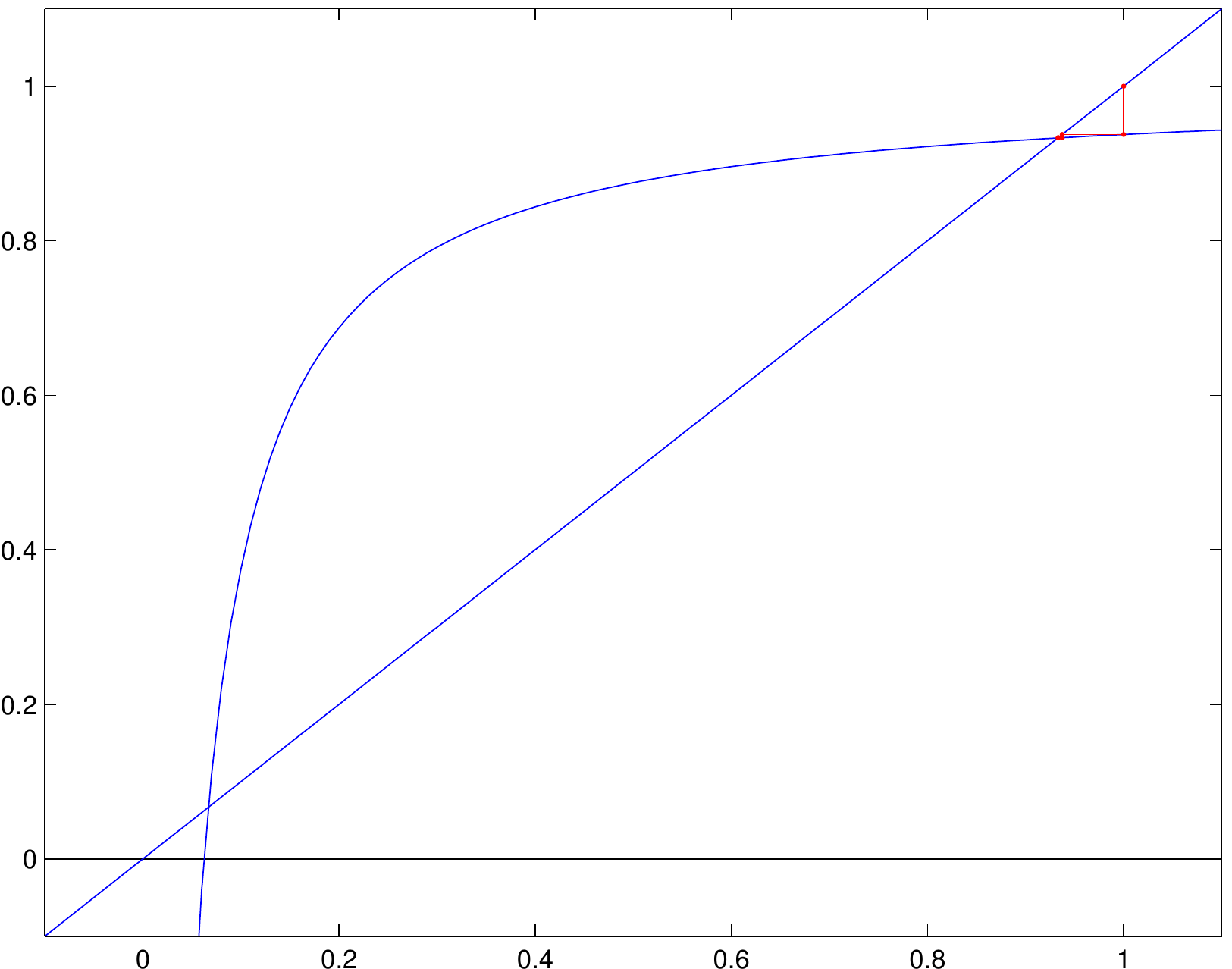}}
\subfigure[$\delta=0.5$]{\includegraphics[width=68mm]{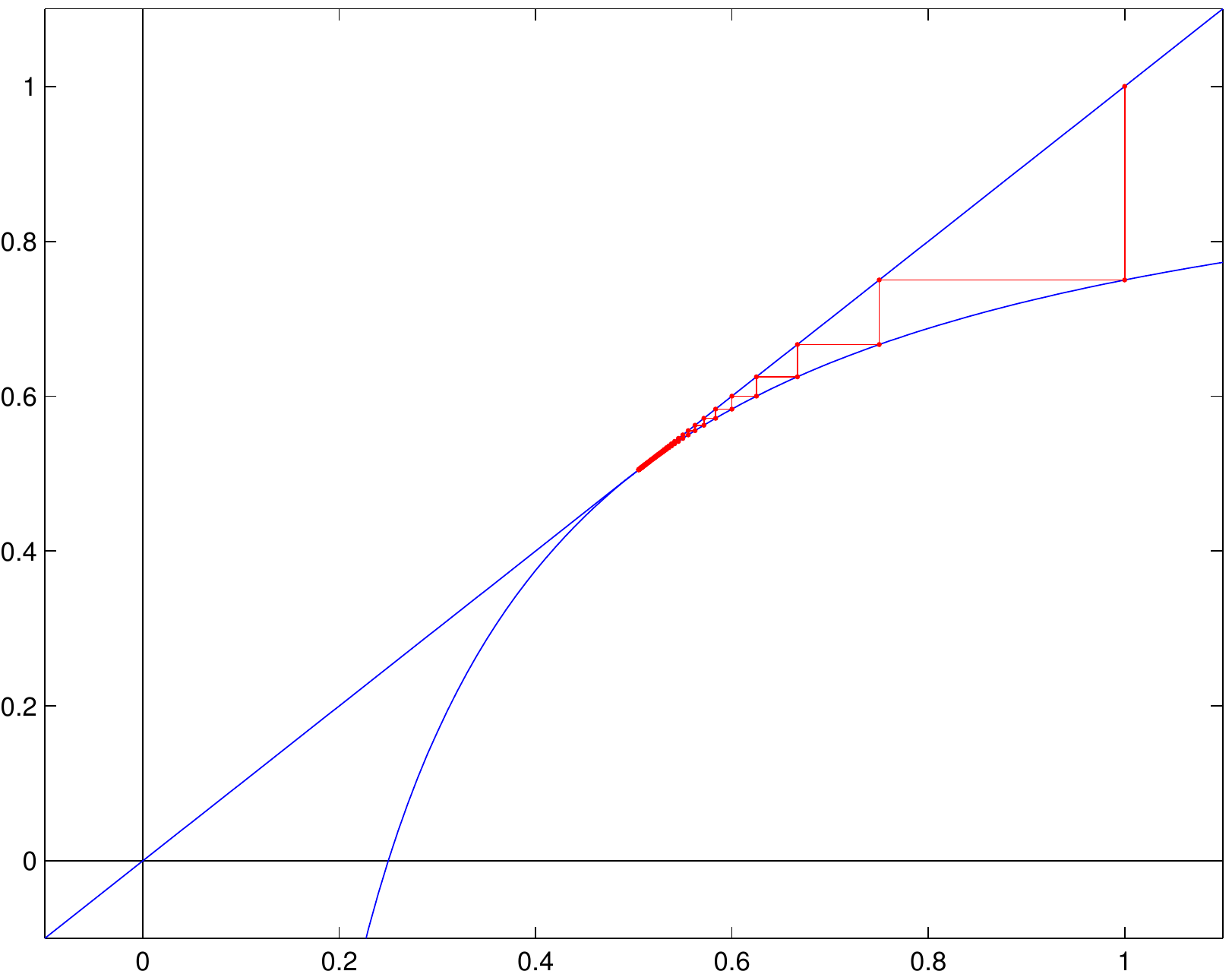}} 
\end{center}
\vspace{-.4cm}
\caption{The iterations of $u_{n+1}=1-\delta ^2/u_n$ with $u_1=1$.}
\label{suite_rec}
\end{figure}

We can then rewrite (\ref{proof:U_I_inv_1}) in order to get
\begin{align}
\label{eq:x_vs_z}
\left\vert x_{m+k}\right\vert &\leq \frac{\left\vert \delta_k\right\vert}{\left\vert \delta_{k+1}\right\vert}\left\vert z_{m+k}\right\vert + \sum_{l=1}^{\infty} \delta^l \frac{\left\vert \delta_k\right\vert\,..\,\left\vert \delta_{k+l}\right\vert}{\left\vert \delta_{k+1}\right\vert\,..\,\left\vert \delta_{k+l+1}\right\vert} \, \left\vert b_{k+1}\right\vert\,..\,\left\vert b_{k+l}\right\vert\, \left\vert z_{m+k+l}\right\vert \nonumber\\
&\leq \frac{\left\vert \delta_k\right\vert}{\left\vert \delta_{k+1}\right\vert}\left\vert z_{m+k}\right\vert + \sum_{l=1}^{\infty} \delta^l \frac{1}{u_{k+1}}\,..\,\frac{1}{u_{k+l}}\frac{\left\vert \delta_{k+l}\right\vert}{\left\vert \delta_{k+l+1}\right\vert} \left\vert z_{m+k+l}\right\vert \nonumber\\
&\leq \sum_{l=0}^{\infty} \left(\frac{\delta}{\gamma}\right)^l \frac{\left\vert \delta_{k+l}\right\vert}{\left\vert \delta_{k+l+1}\right\vert} \left\vert z_{m+k+l}\right\vert \nonumber\\
&\leq \sum_{l=0}^{\infty} \left(\frac{\delta}{\gamma}\right)^l \frac{1}{\gamma\left\vert b_{k+l+1}\right\vert} \left\vert z_{m+k+l}\right\vert\\
&\leq \frac{\left\Vert z_I\right\Vert_s}{C_1\gamma}\sum_{l=0}^{\infty} \left(\frac{\delta}{\gamma}\right)^l \frac{1}{\left(k+l+1\right)^{s_L}\left(m+k+l\right)^{s}}. \nonumber
\end{align}
Finally, since $\displaystyle{\delta<\frac{1}{2}<\gamma}$, 
\begin{equation*}
\left\vert x_{m+k}\right\vert\left(m+k \right)^{s+s_L}\leq \frac{\left\Vert z_I\right\Vert_s}{C_1\gamma}\frac{1}{1-\frac{\delta}{\gamma}}\frac{\left(m+k\right)^{s+s_L}}{\left(k+1\right)^{s_L}\left(m+k\right)^{s}}
\end{equation*}
and $x_I\in\Omega^{s+s_L}$.
\end{proof}

\begin{lemma} \label{lem:L_I_inv}
Assume that $m\geq k_0$ and $\delta<\frac{1}{2}$, Then $L_I^{-1}$ maps $\Omega^s$ into $\Omega^{s}$.
\end{lemma}

\begin{proof}
Let $y_I\in \Omega^s$ and $z_I=L_I^{-1}y_I$. Using (\ref{1i}) and the formula above \big(without the last term since we do not consider here $L_I^{-1}(y_I  -  \lambda_m\, x_{m-1} \be )$\big), we get
\begin{align*}
\label{proof:L_I_inv_1}
\left\vert z_{m+k}\right\vert &\leq  \left\vert y_{m+k}\right\vert + \sum_{l=1}^{k} \frac{\left\vert \delta_{k-l}\right\vert}{\left\vert \delta_{k}\right\vert} \, \left\vert a_{k-l+2}\right\vert\,..\,\left\vert a_{k+1}\right\vert\, \left\vert y_{m+k-l}\right\vert  \\ 
&\leq  \left\vert y_{m+k}\right\vert+\sum_{l=1}^{k} \delta^{l} \frac{\left\vert \delta_{k-l}\right\vert}{\left\vert \delta_{k}\right\vert} \, \left\vert b_{k-l+2}\right\vert\,..\,\left\vert b_{k+1}\right\vert\, \left\vert y_{m+k-l}\right\vert\\
&\leq  \left\vert y_{m+k}\right\vert+\sum_{l=1}^{k} \delta^{l} \frac{\left\vert \delta_{k-l}\right\vert\,..\,\left\vert \delta_{k-1}\right\vert}{\left\vert \delta_{k-l+1}\right\vert\,..\,\left\vert \delta_{k}\right\vert} \,\frac{\left\vert b_{k-l+1}\right\vert\left\vert b_{k-l+2}\right\vert\,..\,\left\vert b_{k+1}\right\vert}{\left\vert b_{k-l+1}\right\vert} \, \left\vert y_{m+k-l}\right\vert\\
&\leq  \left\vert y_{m+k}\right\vert+\sum_{l=1}^{k} \delta^{l} \frac{1}{u_{k-l+1}}\,..\,\frac{1}{u_{k}} \, \frac{\left\vert b_{k+1}\right\vert}{\left\vert b_{k-l+1}\right\vert} \left\vert y_{m+k-l}\right\vert,
\end{align*}
where we use the sequence $u_k$ introduced in the previous proof. We get
\begin{equation}
\label{eq:z_vs_y}
\left\vert z_{m+k}\right\vert \leq \sum_{l=0}^{k} \left(\frac{\delta}{\gamma}\right)^{l} \frac{\left\vert b_{k+1}\right\vert}{\left\vert b_{k-l+1}\right\vert} \left\vert y_{m+k-l}\right\vert,
\end{equation}
and
\begin{align*}
\left\vert z_{m+k}\right\vert\left(m+k\right)^s &\leq  \frac{C_2\left\Vert y\right\Vert_s}{C_1}\sum_{l=0}^{k} \left(\frac{\delta}{\gamma}\right)^{l} \left(\frac{k+1}{k+1-l}\right)^{s_{L}}\left(\frac{m+k}{m+k-l}\right)^{s}\\
&\leq  \frac{C_2\left\Vert y\right\Vert_s}{C_1}\sum_{l=0}^{k} \left(\frac{\delta}{\gamma}\right)^{l} \left(\frac{m+k}{k+1-l}\right)^{s+s_L}.
\end{align*}
For any $k\geq m$, we then have
\begin{align*}
\left\vert z_{m+k}\right\vert\left(m+k\right)^s &\leq  \frac{2^{s+s_L}C_2\left\Vert y\right\Vert_s}{C_1}\left(\sum_{l=0}^{\left[\frac{k}{2}\right]} \left(\frac{\delta}{\gamma}\right)^{l} \left(\frac{k}{k+1-l}\right)^{s+s_L} + \sum_{l=\left[\frac{k}{2}\right]+1}^{k} \left(\frac{\delta}{\gamma}\right)^{l} \left(\frac{k}{k+1-l}\right)^{s+s_L}\right)\\
&\leq  \frac{2^{s+s_L}C_2\left\Vert y\right\Vert_s}{C_1}\left(2^{s+s_L}\sum_{l=0}^{\left[\frac{k}{2}\right]} \left(\frac{\delta}{\gamma}\right)^{l} +  \left(\frac{\delta}{\gamma}\right)^{\frac{k}{2}}\sum_{l=\left[\frac{k}{2}\right]+1}^{k} k^{s+s_L}\right)\\
&\leq  \frac{2^{s+s_L}C_2\left\Vert y\right\Vert_s}{C_1}\left(\frac{2^{s+s_L}}{1-\frac{\delta}{\gamma}} +  \left(\frac{\delta}{\gamma}\right)^{\frac{k}{2}}\frac{k^{s+s_L+1}}{2}\right),
\end{align*}
which is bounded uniformly in $k$ since the last term goes to $0$ when $k$ goes to $\infty$, and the proof is complete.
\end{proof}

\begin{proposition} \label{prop:A_tri_domain_image}
Assume that $m\geq k_0$ and $\delta<\frac{1}{2}$. Then $A$ maps $\Omega^s$ into $\Omega^{s+s_L}$.
\end{proposition}

\begin{proof}
Consider $y=(y_F,y_I)^T \in \Omega^s$. Let $x=(x_F,x_I)^T=Ay$. Then, by definition of the operator $A$ in \eqref{eq:operator_A},
\begin{eqnarray*}
x_F &=&  A_m y_F  - \beta_{m-1}\, \bigg( \bigg\{ (A_m)_{c_{m-1}} \bigg\} \otimes \bigg\{ \left( U_I^{-1} {L_I}^{-1} \right)_{r_0} \bigg\} \bigg) y_I 
\\
&=& A_m y_F  - \beta_{m-1} \left\{ \left( {U_I}^{-1} {L_I}^{-1} \right)_{r_0} y_I  \right\} (A_m)_{c_{m-1}}.
\end{eqnarray*}
By the previous lemmas, ${U_I}^{-1} {L_I}^{-1} y_I \in \Omega^{s+sL}$, and in particular $\left( {U_I}^{-1} {L_I}^{-1} \right)_{r_0} y_I=\left( {U_I}^{-1} {L_I}^{-1} y_I\right)_{0}$ is well defined and so is $x_F$.

Using \eqref{eq:operator_A} again, 
\begin{equation*}
x_I = -\lambda_m \left(\bigg\{ w_I \bigg\} \otimes \, \bigg\{  (A_m)_{r_{m-1}} \bigg\}\right)\,y_F + {U_I}^{-1} {L_I}^{-1}y_I + \Lambda y_I.
\end{equation*}
Remember that $w_I = U_I^{-1} {L_I}^{-1} \be$, so that $w_I\in\Omega^s$ for any $s$. According to the previous lemmas and the definition of $\Lambda$ (see \eqref{eq:operator_A}), we see that $x_I \in \Omega^{s+s_L}$.
\end{proof}

\section{Computations of fixed points of the operator \boldmath $T$ \unboldmath}
\label{sec:rigorous_computational method}

Our main motivation for computing approximate inverses is to prove existence, in a mathematically rigorous sense, of a fixed point of the Newton-like operator $T$ in a set centered at a numerical approximation $\bx$. The Newton-like operator has the form
\begin{equation} \label{eq:T}
T(x)=x-Af(x),
\end{equation}
where $A$ is the approximate inverse \eqref{eq:operator_A} of $Df(\bx)$ computed using the theory 
of Section~\ref{sec:tridiagonal}. Since $f$ maps $\Omega^s$ into $\Omega^{s-s_L}$ and $A$ maps $\Omega^s$ into $\Omega^{s+s_L}$ 
(thanks to Proposition~\ref{prop:A_tri_domain_image}), we see that $T$ maps
 the Banach space $\Omega^s$ into itself. Our goal is to obtain explicit bounds allowing us to show that a given $T$ is a contraction on the ball $B_{\bx}(r)$,
 which yields the existence of a fixed point of $T$ (and thus of a zero of $f$). The fixed  point theorem that we use (see Theorem~\ref{thm:Tfixedpt}) requires bounds on $T$ and its derivative.
 We get formulas for these bounds in Sections~\ref{sec:Y} and \ref{sec:Z}, and then explain in Section~\ref{sec:radii_pol} how to use the 
so-called radii polynomials in order 
to find a radius $r>0$ such that $T(B_{\bx}(r)) \subset B_{\bx}(r)$,
 and such that $T$ is a contraction on $B_{\bx}(r)$.

Before proceeding further, we endow $\Omega^s$ with the operation of discrete convolution.
 More precisely, given $x=(x_k)_{k \ge 0},y=(y_k)_{k \ge 0} \in \Omega^s$, we 
extend $x,y$ symmetrically by $\tx=(x_k)_{k \in \Z}, \ty=(y_k)_{k \in \Z}$ where $\tx_{-k}=x_k$, 
$\ty_{-k}=y_k$, for $k \ge 1$. The discrete convolution of $x$ and $y$ is then denoted by $x * y$, and defined by the (infinite) sum
\[
(x*y)_k = \sum_{\stackrel{k_1+k_2 = k}{k_1,k_2 \in \Z}}  \tx_{k_1} \ty_{k_2}.
\]
It is known that for $s >1$, $\left( \Omega^s, * \right)$ is an algebra  (e.g. see \cite{MR3125637}), that is, if $x,y \in \Omega^s$, then $x*y \in \Omega^s$.
This will be useful when we shall look for a bound such as \eqref{eq:Z} below. We start with a classical theorem, whose proof 
is standard (e.g. see the proof of Lemma 3.3 in \cite{MR2718657}) and is a direct consequence of the contraction mapping theorem.

\begin{theorem}
\label{thm:Tfixedpt}
For a given $s > 1$, consider $T\colon \Omega^s\to\Omega^s$ with $T=(T_k)_{k \ge 0}$, $T_k \in \R$.  Assume that there exists a point $\bx\in\Omega^s$ and vectors $Y=\{Y_k\}_{k \ge 0}$ and $Z=\{Z_k(r)\}_{k \ge 0}$, with $Y_k, Z_k(r) \in \R$, satisfying (for all $k\ge 0$)
\begin{equation} \label{eq:Y}
|(T(\bar x)-\bar x)_k| \le Y_{k},
\end{equation}
and
\begin{equation} \label{eq:Z}
\sup_{b_1,b_2\in B_0(r)}\Big|\big[DT(\bar x+b_1)b_2\big]_{k}\Big| \le Z_{k}(r).
\end{equation}
If there exists $r>0$ such that $\| Y+Z(r)\|_s<r$, then the operator $T$ is a contraction in $B_{\bx}(r)$ and there exists a unique $\hx \in B_{\bx}(r)$ such that $T(\hx)=\hx$.
\end{theorem}

We shall see how to get the bounds $Y$ (Section~\ref{sec:Y}) and the bounds $Z(r)$ (Section~\ref{sec:Z}), and we shall
 provide  an efficient way of finding a radius $r>0$ such that $\| Y+Z(r)\|_s<r$ (Section~\ref{sec:radii_pol}). The first step
however consists in looking for bounds on $A$. More precisely, we need some estimates in order to control the action of $U_I^{-1}L_I^{-1}$.
This is the goal of the following Subsection.

\subsection{Some preliminary computations}

We introduce the notations 
\begin{equation} \label{eq:theta_eta}
\theta \bydef \frac{\delta}{\gamma} ~~~{\rm and} ~~~ \eta \bydef  \frac{1}{\gamma(1-\theta^2)}.
\end{equation}

\begin{lemma}\label{lem:x_vs_y}
Let $y_I=\left(y_m,y_{m+1},\ldots\right)^T$ be an infinite vector and $x_I=U_I^{-1}L_I^{-1}y_I$. Assume that $m\geq k_0$ and $\delta<\frac{1}{2}$. Then, for all $k\geq 0$,  
\begin{equation*}
\left\vert x_{m+k}\right\vert \leq \eta \left(\sum_{j=0}^{k}\theta^{k-j}\frac{\left\vert y_{m+j}\right\vert}{\left\vert \mu_{m+j}\right\vert} + \sum_{j=k+1}^{\infty}\theta^{j-k}\frac{\left\vert y_{m+j}\right\vert}{\left\vert \mu_{m+j}\right\vert}\right).
\end{equation*}
\end{lemma}
\begin{proof} We 
again introduce $z_I=L_I^{-1}y_I$. Combining \eqref{eq:x_vs_z} from Lemma~\ref{lem:U_I_inv} and \eqref{eq:z_vs_y} from Lemma~\ref{lem:L_I_inv}, we get
\begin{align*}
\left\vert x_{m+k}\right\vert & \leq  \frac{1}{\gamma}\sum_{l=0}^{\infty}\sum_{j=0}^{k+l}\theta^{k+2l-j}\frac{\left\vert y_{m+j}\right\vert}{\left\vert b_{j+1}\right\vert}\\
& = \frac{1}{\gamma}\left(\sum_{j=0}^{k}\frac{\left\vert y_{m+j}\right\vert}{\left\vert b_{j+1}\right\vert}\sum_{l=0}^{\infty}\theta^{k+2l-j} + \sum_{j=k+1}^{\infty}\frac{\left\vert y_{m+j}\right\vert}{\left\vert b_{j+1}\right\vert}\sum_{l=j-k}^{\infty}\theta^{k+2l-j}\right)\\
& =  \frac{1}{\gamma}\left(\sum_{j=0}^{k}\frac{\left\vert y_{m+j}\right\vert}{\left\vert b_{j+1}\right\vert}\frac{\theta^{k-j}}{1-\theta^2} + \sum_{j=k+1}^{\infty}\frac{\left\vert y_{m+j}\right\vert}{\left\vert b_{j+1}\right\vert}\frac{\theta^{j-k}}{1-\theta^2}\right)\\
& = \eta \left(\sum_{j=0}^{k}\theta^{k-j}\frac{\left\vert y_{m+j}\right\vert}{\left\vert \mu_{m+j}\right\vert} + \sum_{j=k+1}^{\infty}\theta^{j-k}\frac{\left\vert y_{m+j}\right\vert}{\left\vert \mu_{m+j}\right\vert}\right). \qedhere
\end{align*}
\end{proof}
In particular, we immediately obtain  the two following corollaries (always under the assumptions of Lemma~\ref{lem:x_vs_y}) which will be useful in the sequel.
\begin{corollary}
Recall \eqref{eq:theta_eta}. Then, for $w_I=\left(w_m,w_{m+1},\ldots\right)^T \bydef U_I^{-1} L_I^{-1}\be$, we have
\begin{equation}\label{eq:maj_w_I}
\left\vert w_{m+k}\right\vert \leq \eta\theta^{k}\frac{1}{\left\vert\mu_m\right\vert},~~~~~ {\rm \it for~all~} k\geq 0.
\end{equation}
\end{corollary}
\begin{corollary}
If $y$ is such that $y_{m+k}=0$ for any $k\geq n$, then
\begin{equation}\label{eq:maj_T_inv_F1}
\forall ~ k\leq n-2,\quad \left\vert x_{m+k}\right\vert \leq \eta \left(\sum_{l=0}^{k}\theta^{k-l}\frac{\left\vert y_{m+l}\right\vert}{\left\vert \mu_{m+l}\right\vert} +\sum_{l=k+1}^{n-1}\theta^{l-k}\frac{\left\vert y_{m+l}\right\vert}{\left\vert \mu_{m+l}\right\vert}\right)
\end{equation}
and
\begin{equation}\label{eq:maj_T_inv_F2}
\forall ~ k\geq n-1,\quad \left\vert x_{m+k}\right\vert \leq \eta \theta^k\sum_{l=0}^{n-1}\frac{\left\vert y_{m+l}\right\vert}{\theta^{l}\left\vert \mu_{m+l}\right\vert} .
\end{equation}
\end{corollary}

More generally, we will also need in the next two Subsections a uniform bound 
on $\left\vert x_{m+k}\right\vert (m+k)^{s+s_L}$ for $k$ large enough. We assume here that $m\geq 2$ (which will always be the case in practice), and define for any integer $M$
\begin{equation*}
\chi=\chi(\theta,m,M,s,s_L) \bydef \theta^{\frac{M}{2}}\frac{M}{2}\left(\frac{m+M}{m}\right)^{s+s_L} + \theta^{\sqrt M}\frac{M}{2}2^{s+s_L}+ \frac{1}{1-\theta}\left(\frac{m+M}{m+M-\sqrt M-1}\right)^{s+s_L}.
\end{equation*}
\begin{proposition}
Suppose that $M$ satisfies
\begin{equation}\label{eq:M_A}
M\geq \max\left(\frac{-m\ln\sqrt \theta -s-s_L-1-\sqrt{(m\ln\sqrt\theta+s+s_L+1)^2-4m\ln\sqrt\theta}}{2\ln \sqrt \theta},\frac{4}{\left(\ln\theta\right)^2},m\right).
\end{equation}
Then for all $k< M$,
\begin{equation}\label{eq:maj_unif_expl}
\left\vert x_{m+k}\right\vert\left(m+k\right)^{s+s_L} \leq \frac{\eta \Vert y_I \Vert_s}{C_1} \left(\sum_{l=0}^{k}\theta^{k-l}\left(\frac{m+k}{m+l}\right)^{s+s_L}  + \frac{\theta}{1-\theta}\right),
\end{equation}
and for all $k\geq M$
\begin{equation}\label{eq:maj_unif_tail}
\left\vert x_{m+k}\right\vert\left(m+k\right)^{s+s_L} \leq \frac{\eta \Vert y_I \Vert_s}{C_1} \left(\chi  + \frac{\theta}{1-\theta}\right).
\end{equation}
\end{proposition}
\begin{proof}
Thanks to Lemma~\ref{lem:x_vs_y},
\begin{align*}
\left\vert x_{m+k}\right\vert\left(m+k\right)^{s+s_L} & \leq  \frac{\eta \Vert y_I \Vert_s}{C_1} \left( \sum_{l=0}^{k}\theta^{k-l}\left(\frac{m+k}{m+l}\right)^{s+s_L} + \sum_{l=k+1}^{\infty}\theta^{l-k}\left(\frac{m+k}{m+l}\right)^{s+s_L}\right)\\
& \leq  \frac{\eta \Vert y_I \Vert_s}{C_1} \left(\sum_{l=0}^{k}\theta^{k-l}\left(\frac{m+k}{m+l}\right)^{s+s_L}  + \frac{\theta}{1-\theta}\right).
\end{align*} 
Then for $k\geq M$, we split the remaining sum
\begin{align*}
\sum_{l=0}^{k}\theta^{k-l}\left(\frac{m+k}{m+l}\right)^{s+s_L}  &= \sum_{l=0}^{\left[\frac{k}{2}\right]-1}\theta^{k-l}\left(\frac{m+k}{m+l}\right)^{s+s_L} + \sum_{l=\left[\frac{k}{2}\right]}^{k-\left[\sqrt k\right]-1}\theta^{k-l}\left(\frac{m+k}{m+l}\right)^{s+s_L}+ \sum_{l=k-\left[\sqrt k\right]}^{k}\theta^{k-l}\left(\frac{m+k}{m+l}\right)^{s+s_L} \\
&\leq \theta^{\frac{k}{2}}\frac{k}{2}\left(\frac{m+k}{m}\right)^{s+s_L} + \theta^{\sqrt k}\frac{k}{2}2^{s+s_L} + \frac{1}{1-\theta}\left(\frac{m+k}{m+k-\sqrt k -1}\right)^{s+s_L} \\
&\leq \theta^{\frac{M}{2}}\frac{M}{2}\left(\frac{m+M}{m}\right)^{s+s_L} + \theta^{\sqrt M}\frac{M}{2}2^{s+s_L} + \frac{1}{1-\theta}\left(\frac{m+M}{m+M-\sqrt M-1}\right)^{s+s_L}\\
& = \chi.
\end{align*}
The justification of the last inequality is contained in the following three lemmas.
\end{proof}
\begin{lemma}
If $M$ satisfies \eqref{eq:M_A}, then for all $k\geq M$
\begin{equation*}
\theta^{\frac{k}{2}}\frac{k}{2}\left(\frac{m+k}{m}\right)^{s+s_L} \leq \theta^{\frac{M}{2}}\frac{M}{2}\left(\frac{m+M}{m}\right)^{s+s_L}.
\end{equation*}
\end{lemma}
\begin{proof}
For $x>0$, let $\displaystyle \varphi_1(x) \bydef \theta^{\frac{x}{2}}x(m+x)^{s+s_L}$, whose derivative is
\begin{align*}
\varphi_1'(x) &= \sqrt{\theta}^x\left(\left(\ln\sqrt\theta\right) x(m+x)^{s+s_L}+\left(m+x\right)^{s+s_L}+(s+s_L)x\left(m+x\right)^{s+s_L-1}\right)\\
&= \left(m+x\right)^{s+s_L-1}\sqrt{\theta}^x\left(\left(\ln\sqrt\theta\right)(m+x)x +(m+x)+(s+s_L)x\right)\\
&= \left(m+x\right)^{s+s_L-1}\sqrt{\theta}^x\left(\left(\ln\sqrt\theta\right) x^2 + \left(m\ln\sqrt\theta+s+s_L+1\right)x + m\right).
\end{align*}
For $0<\theta<1$, the discriminant of $\ln\sqrt\theta x^2 + \left(m\ln\sqrt\theta+s+s_L+1\right)x + m$ given by
\begin{equation*}
\Delta \bydef \left(m\ln\sqrt\theta +s+s_L+1\right)^2-4m\ln\sqrt\theta,
\end{equation*}
is positive.  Since $M$ satisfies (\ref{eq:M_A}), $\varphi_1'(x)\leq 0$ for any $x\geq M$ and so $\varphi_1(k)\leq\varphi_1(M)$ for all $k\geq M$.
\end{proof}
\begin{lemma}
If $M$ satisfies \eqref{eq:M_A}, then for all $k\geq M$,
\begin{equation*}
\theta^{\sqrt k}\frac{k}{2}2^{s+s_L} \leq \theta^{\sqrt M}\frac{M}{2}2^{s+s_L}.
\end{equation*}
\end{lemma}
\begin{proof}
Let $\varphi_2(x) \bydef \theta^{\sqrt x}x$. Then 
\[
\varphi_2'(x) = \theta^{\sqrt x}\left(\frac{\ln\theta}{2\sqrt x}x+1\right) = \frac{\theta^{\sqrt x}}{2}\left(\sqrt x\ln\theta+2\right).
\]
Hence, for $x\geq \displaystyle \frac{4}{(\ln\theta)^2}$, $\varphi_2'(x)\leq 0$ and so $\varphi_2(k)\leq\varphi_2(M)$ for all $k\geq M$.
\end{proof}
\begin{lemma}
If $M$ satisfies \eqref{eq:M_A}, then for all $k\geq M$,
\begin{equation*}
\frac{1}{1-\theta}\left(\frac{m+k}{m+k-\sqrt k -1}\right)^{s+s_L} \leq \frac{1}{1-\theta}\left(\frac{m+M}{m+M-\sqrt M-1}\right)^{s+s_L}.
\end{equation*}
\end{lemma}
\begin{proof}
Let $\displaystyle \varphi_3(x) \bydef \frac{m+x}{m+x-\sqrt x -1}$. Then
\[
\varphi_3'(x) = \frac{m+x-\sqrt x -1 -(m+x)\left(1-\frac{1}{2\sqrt x}\right)}{\left(m+x-\sqrt x -1\right)^2}
= -\frac{x+2\sqrt x-m}{2\sqrt x\left(m+x-\sqrt x -1\right)^2}.
\]
Hence, for $x\geq m$, $\varphi_3'(x)\leq 0$ and $\varphi_3(k)\leq\varphi_3(M)$ for all $k\geq M$.
\end{proof}

Finally, we will need to bound the error made by using $\tilde w$ instead of $\left(w_I\right)_0$ for the definition (\ref{eq:operator_A}) of $A$. 
\begin{lemma}
Assume that $L\geq k_0$ and $\delta<\frac{1}{2}$. Then
\begin{equation}
\label{eq:err_tilde_w}
\left\vert \left(w_I\right)_0 - \tilde w\right\vert \leq  \frac{\theta^{2L}}{\left\vert \mu_m\right\vert(1-\theta^2)}.
\end{equation}
\end{lemma}
\begin{proof}
Using (\ref{eq:asump2_tridiag}) together with the sequence $(u_l)$ introduced in the proof of Lemma~\ref{lem:U_I_inv}, we get
\begin{align*}
\left\vert \left(w_I\right)_0 - \tilde w\right\vert & \leq \sum_{l=L}^{\infty}\frac{\left\vert \delta_0\right\vert^2}{\left\vert \delta_l\right\vert\left\vert \delta_{l+1}\right\vert}\left\vert c_1\right\vert\ldots \left\vert c_l\right\vert\, \left\vert a_2\right\vert\ldots \left\vert a_{l+1}\right\vert \\
& \leq \frac{\left\vert \delta_0\right\vert}{\left\vert \delta_1\right\vert}\sum_{l=L}^{\infty}\delta^{2l}\left(\frac{1}{u_1}\cdots\frac{1}{u_l}\right)\left(\frac{1}{u_2}\cdots\frac{1}{u_{l+1}}\right) \\
& \leq \frac{1}{\left\vert \mu_m\right\vert}\sum_{l=L}^{\infty}\theta^{2l} \\
& = \frac{\theta^{2L}}{\left\vert \mu_m\right\vert(1-\theta^2)}. \qedhere
\end{align*}
\end{proof}

\subsection{Computation of the  \boldmath $Y$ \unboldmath bounds}
\label{sec:Y}

From now on, we shall assume for the sake of clarity that the nonlinearity $N$ of $f$ in \eqref{eq:general_f=0} is a polynomial of degree two.
The generalization to a polynomial nonlinearity of higher degree could be obtained thanks
to the use of the estimates developed in \cite{MR2718657} in order to bound terms like
\begin{equation*}
\left(x^1\ast\ldots\ast x^p\right)_n
\end{equation*}
where $x^1,\ldots,x^p\in B_{0}(r)$. Moreover, as long as one is interested in 
problems with nonlinearities built from elementary functions of mathematical
physics (powers, exponential, trigonometric functions, rational, Bessel, elliptic integrals, etc.), 
our method is applicable.
Indeed, since these nonlinearities are themselves solutions of low order 
linear or polynomial ODEs, they can be appended to the original problem of
interest in order to obtain  polynomial nonlinearities, albeit in a higher number of 
variables. This standard trick is explained in more details in \cite{MR633878}, and is used in \cite{LMR} to prove existence of periodic solutions 
in the planar circular restricted three body problem.

With this in mind, we are ready to compute the bound $Y$ appearing in
 Theorem~\ref{thm:Tfixedpt}. In everything that follows, $\left\vert \cdot \right\vert$, when applied to vectors or matrices
 (even infinite dimensional), must be understood component-wise.
\medskip

The main estimate of this subsection, that is the bound on $Y$, is presented in the following Proposition:
\medskip

\begin{proposition}\label{prop:Y}
Consider an integer $M$ such that
\begin{equation}\label{eq:M_Y}
M\geq \max\left(\frac{-s}{\ln\theta}-m,m-2\right),
\end{equation}
and define $Y=(Y_k)_{k \ge 0}$ component-wise by
\begin{equation}
\label{eq:Y_F}
Y_F \bydef \left\vert A_m\left(f(\bx)\right)_F\right\vert + \left\vert \beta_{m-1}\right\vert\eta\left(\sum_{l=0}^{m-2}\theta^l\frac{\left\vert f(\bx)\right\vert_{m+l}}{\left\vert \mu_{m+l}\right\vert}\right) \left\vert\left( A_m\right)_{c_{m-1}}\right\vert,
\end{equation}
\begin{align}
\label{eq:Y_k_1}
Y_{m+k} &\bydef \left(\left\vert\left(A_m\right)_{r_{m-1}} f(\bx)_F\right\vert + \left\vert \beta_{m-1} \left(A_m\right)_{m-1,m-1}\right\vert \eta\left(\sum_{l=0}^{m-2}\theta^l\frac{\left\vert f(\bx)\right\vert_{m+l}}{\left\vert \mu_{m+l}\right\vert}\right)\right)\eta\theta^{k}\frac{\left\vert \lambda_m\right\vert}{\left\vert\mu_m\right\vert} \nonumber\\
&+ \eta\sum_{l=0}^{k}\theta^{k-l}\frac{\left\vert f(\bx)\right\vert_{m+l}}{\left\vert \mu_{m+l}\right\vert} + \eta\sum_{l=k+1}^{m-2}\theta^{l-k}\frac{\left\vert f(\bx)\right\vert_{m+l}}{\left\vert \mu_{m+l}\right\vert},\quad \forall ~0\leq k\leq m-3, 
\end{align}
\begin{align}
\label{eq:Y_k_2}
Y_{m+k} &\bydef \left(\left\vert\left(A_m\right)_{r_{m-1}} f(\bx)_F\right\vert + \left\vert \beta_{m-1} \left(A_m\right)_{m-1,m-1}\right\vert \eta\left(\sum_{l=0}^{m-2}\theta^l\frac{\left\vert f(\bx)\right\vert_{m+l}}{\left\vert \mu_{m+l}\right\vert}\right)\right)\eta\theta^{k}\frac{\left\vert \lambda_m\right\vert}{\left\vert\mu_m\right\vert} \nonumber \\
&+ \eta \theta^k\sum_{l=0}^{m-2}\frac{\left\vert f(\bx)\right\vert_{m+l}}{\theta^{l}\left\vert \mu_{m+l}\right\vert},\quad \forall m-2\leq k\leq M, 
\end{align}
and
\begin{equation}
 \label{eq:Y_tail}
Y_{m+k}\bydef Y_{m+M}\frac{\omega_{m+M}^s}{\omega_{m+k}^s},\quad \forall~ k>M.
\end{equation}
Then
\begin{equation*}
\vert T(\bx) - \bx \vert \leq Y.
\end{equation*}
\end{proposition}
\begin{proof}
By definition of $T$,
\begin{equation*}
\left\vert T(\bx)-\bx\right\vert = \left\vert Af(\bx)\right\vert.
\end{equation*}
Note that since we suppose that $f$ is at most quadratic,
 and since $\bx$ is constructed in such a way that $\bx_k=0$ for all $k\geq m$, we get the identity
 $\left(f(\bx)\right)_{m+k}=0$ for all $k\geq m-1$.
Thanks to (\ref{eq:operator_A}),
\begin{equation*}
\left\vert \left(Af(\bx)\right)_F\right\vert \leq \left\vert A_m\left(f(\bx)\right)_F\right\vert + \left\vert \beta_{m-1}\right\vert\left\vert \left(U_I^{-1}L_I^{-1}\left(f(\bx)\right)_I\right)_0\right\vert \left\vert\left( A_m\right)_{c_{m-1}}\right\vert,
\end{equation*}
so that using (\ref{eq:maj_T_inv_F1}) with $n=m-1$ and $k=0$, we get 
\begin{equation*}
\left\vert \left(Af(\bx)\right)_F\right\vert \leq \left\vert A_m\left(f(\bx)\right)_F\right\vert + \left\vert \beta_{m-1}\right\vert\eta\left(\sum_{l=0}^{m-2}\theta^l\frac{\left\vert f(\bx)\right\vert_{m+l}}{\left\vert \mu_{m+l}\right\vert}\right) \left\vert\left( A_m\right)_{c_{m-1}}\right\vert ,
\end{equation*}
which provides the bound \eqref{eq:Y_F}. 

Using (\ref{eq:operator_A}) again,
\begin{equation*}
\left\vert \left(Af(\bx)\right)_I\right\vert \leq \left\vert\lambda_m\right\vert \left(\left\vert\left(A_m\right)_{r_{m-1}} f(\bx)_F\right\vert + \left\vert \beta_{m-1} \left(A_m\right)_{m-1,m-1} \left(U_I^{-1} L_I^{-1} f(\bx)_I\right)_0\right\vert\right)\left\vert w_I\right\vert + \left\vert U_I^{-1} L_I^{-1} f(\bx)_I\right\vert,
\end{equation*}
so using (\ref{eq:maj_w_I}), (\ref{eq:maj_T_inv_F1}) and (\ref{eq:maj_T_inv_F2}) (again with $n=m-1$), we get
\begin{align*}
\left\vert \left(Af(\bx)\right)_{m+k}\right\vert &\leq \left(\left\vert\left(A_m\right)_{r_{m-1}} f(\bx)_F\right\vert + \left\vert \beta_{m-1} \left(A_m\right)_{m-1,m-1}\right\vert \eta\left(\sum_{l=0}^{m-2}\theta^l\frac{\left\vert f(\bx)\right\vert_{m+l}}{\left\vert \mu_{m+l}\right\vert}\right)\right)\eta\theta^{k}\frac{\left\vert \lambda_m\right\vert}{\left\vert\mu_m\right\vert} \nonumber\\
&+ \eta\sum_{l=0}^{k}\theta^{k-l}\frac{\left\vert f(\bx)\right\vert_{m+l}}{\left\vert \mu_{m+l}\right\vert} + \eta\sum_{l=k+1}^{m-2}\theta^{l-k}\frac{\left\vert f(\bx)\right\vert_{m+l}}{\left\vert \mu_{m+l}\right\vert},\quad \forall ~0\leq k\leq m-3, 
\end{align*}
which provides the bound \eqref{eq:Y_k_1}, and
\begin{align*}
\left\vert \left(Af(\bx)\right)_{m+k}\right\vert &\leq \left(\left\vert\left(A_m\right)_{r_{m-1}} f(\bx)_F\right\vert + \left\vert \beta_{m-1} \left(A_m\right)_{m-1,m-1}\right\vert \eta\left(\sum_{l=0}^{m-2}\theta^l\frac{\left\vert f(\bx)\right\vert_{m+l}}{\left\vert \mu_{m+l}\right\vert}\right)\right)\eta\theta^{k}\frac{\left\vert \lambda_m\right\vert}{\left\vert\mu_m\right\vert} \nonumber \\
&+ \eta \theta^k\sum_{l=0}^{m-2}\frac{\left\vert f(\bx)\right\vert_{m+l}}{\theta^{l}\left\vert \mu_{m+l}\right\vert},\quad \forall k\geq m-2,
\end{align*}
which provides the bound \eqref{eq:Y_k_2}. Finally, by \eqref{eq:M_Y}, $\theta^k (m+k)^s \leq \theta^M (m+M)^s$ for all $k> M$, and we obtain the bound \eqref{eq:Y_tail}.
\end{proof}
We present  in Section \ref{sec:radii_pol} the rationale behind the definition of $Y_{m+k}$ for $k>M$.

\subsection{Computation of the  \boldmath $Z$ \unboldmath bounds}
\label{sec:Z}

In order to compute the $Z$ bounds from Theorem~\ref{thm:Tfixedpt}, we need to estimate the quantity
\begin{equation*}
DT\left(\bx +y\right)z=\left(I-ADf\left(\bx+y\right)\right)z=\left(I-AA^{\dag}\right)z-A\left(Df\left(\bx+y\right)-A^{\dag}\right)z
\end{equation*}
for all $y,z\in B_{0}(r)$. We are going to bound each term separately in the next two Sub-subsections. We introduce the notation 
\begin{equation} \label{eq:W^s}
W_F^s \bydef \left(\frac{1}{\omega_0^s},\ldots,\frac{1}{\omega_{m-1}^s}\right)^T.
\end{equation}

\subsubsection{Estimates for \boldmath $(I-A A^\dagger)z$ \unboldmath}\label{sec:Z^1}

In this Sub-subsection, we present the bound on $(I-A A^\dagger)z$, which constitutes the first part of 
a bound for $Z$.

\begin{proposition}\label{prop:Z1}
Let $M$ be an integer satisfying \eqref{eq:M_Y}. We define $Z^1 = (Z^1_k)_{k \ge 0}$ component-wise by
\begin{equation}
\label{eq:Z1_F}
Z^1_F \bydef \left(\left\vert I-A_m\tilde K\right\vert W_F^s + \frac{\left\vert \beta_{m-1}\right\vert\left\vert \lambda_m\right\vert\theta^{2L}}{\left\vert \mu_m\right\vert\omega^s_{m-1}(1-\theta^2)} \left\vert A_m\right\vert_{c_{m-1}}\right)r,
\end{equation}
\begin{equation}
\label{eq:Z1_k}
Z^1_{m+k} \bydef \left(\left\vert I -A_m \tilde K\right\vert_{r_{m-1}}W_F^s + \frac{\left\vert \beta_{m-1}\right\vert\left\vert \lambda_m\right\vert\theta^{2L}}{\left\vert \mu_m\right\vert\omega^s_{m-1}(1-\theta^2)} \left\vert A_m\right\vert_{m-1, m-1}\right)\eta\theta^{k}\frac{\left\vert \lambda_m\right\vert}{\left\vert\mu_m\right\vert}r,\quad \forall~ 0\leq k\leq M,
\end{equation}
and 
\begin{equation}
\label{eq:Z1_tail}
Z^1_{m+k} \bydef Z^1_{m+M}\frac{\omega^s_{m+M}}{\omega^s_{m+k}},\quad \forall ~k>M.
\end{equation}
Then for all $ z\in B_{0}(r)$,
\begin{equation*}
\left\vert \left(I-AA^{\dag}\right)z\right\vert \leq Z^1.
\end{equation*}
\end{proposition}
\begin{proof}

Thanks to (\ref{eq:opertor_A_dag}) and (\ref{eq:operator_A}),
\begin{align*}
\left(AA^{\dag}z\right)_F &= A_m\left(Dz_F+ \begin{pmatrix} 0\\ \vdots \\ \beta_{m-1}z_m\end{pmatrix}\right) - \beta_{m-1}\left(U_I^{-1}L_I^{-1}\left(Tz_I+\lambda_m z_{m-1} \be \right)\right)_0 \left(A_m\right)_{c_{m-1}}\\
&= A_m D z_F + \beta_{m-1}z_m \left(A_m\right)_{c_{m-1}} - \beta_{m-1}\left(z_m+\lambda_m z_{m-1}\left(w_I\right)_0\right) \left(A_m\right)_{c_{m-1}}\\
&= A_m \tilde K z_F + \beta_{m-1}\lambda_m\left(\tilde w -\left(w_I\right)_0\right)z_{m-1}\left(A_m\right)_{c_{m-1}},
\end{align*}
and so
\begin{equation*}
\left(\left(I-AA^{\dag}\right)z\right)_F=\left(I-A_m\tilde K\right)z_F + \beta_{m-1}\lambda_m\left(\tilde w -\left(w_I\right)_0\right)z_{m-1}\left(A_m\right)_{c_{m-1}}.
\end{equation*}
For $z\in B_{0}(r)$ we have, using (\ref{eq:err_tilde_w}),
\begin{equation*}
\left\vert\left(I-AA^{\dag}\right)z\right\vert_F \leq \left(\left\vert I-A_m\tilde K\right\vert W_F^s + \frac{\left\vert \beta_{m-1}\right\vert\left\vert \lambda_m\right\vert\theta^{2L}}{\left\vert \mu_m\right\vert\omega^s_{m-1}(1-\theta^2)} \left\vert A_m\right\vert_{c_{m-1}}\right)r,
\end{equation*}
which provides the bound \eqref{eq:Z1_F}.

Using again (\ref{eq:opertor_A_dag}) and (\ref{eq:operator_A}), we get
\begin{align*}
\left(AA^{\dag}z\right)_I &= -\lambda_m\left(A_m\right)_{r_{m-1}}\left(Dz_F+ \begin{pmatrix} 0\\ \vdots \\ \beta_{m-1}z_m\end{pmatrix}\right)w_I + \left(U_I^{-1}L_I^{-1}+\Lambda\right)\left(Tz_I+\lambda_m z_{m-1}\be \right)\\
&=  z_I + \lambda_m w_I\\
& \left(-\left(A_m\right)_{r_{m-1}}Dz_F - \beta_{m-1}\left(A_m\right)_{m-1,m-1}z_m + z_{m-1} + \beta_{m-1}\left(A_m\right)_{m-1,m-1}\left(z_I+\lambda_mz_{m-1}w_I\right)_0 \right)\\
&=  z_I + \lambda_m\left(-\left(A_m\right)_{r_{m-1}}Dz_F + z_{m-1} + \beta_{m-1}\lambda_m\left(A_m\right)_{m-1,m-1}z_{m-1}\left(w_I\right)_0 \right)w_I\\
&=  z_I + \lambda_m\left(z_{m-1} -\left(A_m\right)_{r_{m-1}}\tilde Kz_F + \beta_{m-1}\lambda_m\left(A_m\right)_{m-1,m-1}z_{m-1}\left(\tilde w -\left(w_I\right)_0\right) \right)w_I\\
&=  z_I + \lambda_m\left(\left(I -A_m \tilde K\right)_{r_{m-1}}z_F + \beta_{m-1}\lambda_m\left(A_m\right)_{m-1,m-1}z_{m-1}\left(\tilde w -\left(w_I\right)_0\right)\right)w_I,\\
\end{align*}
and so
\begin{equation*}
\left(\left(I-AA^{\dag}\right)z\right)_I=-\lambda_m\left(\left(I -A_m K\right)_{r_{m-1}}z_F + \beta_{m-1}\lambda_m\left(A_m\right)_{m-1,m-1}z_{m-1}\left(\tilde w -\left(w_I\right)_0\right)\right)w_I.
\end{equation*}
  
For $z\in B_{0}(r)$ we have, using (\ref{eq:maj_w_I}) and (\ref{eq:err_tilde_w}),
\begin{equation*}
\left\vert \left(I-AA^{\dag}\right)z\right\vert_{m+k} \leq \left(\left\vert I -A_m \tilde K\right\vert_{r_{m-1}}W_F^s + \frac{\left\vert \beta_{m-1}\right\vert\left\vert \lambda_m\right\vert\theta^{2L}}{\left\vert \mu_m\right\vert\omega^s_{m-1}(1-\theta^2)} \left\vert A_m\right\vert_{m-1, m-1}\right)\eta\theta^{k}\frac{\left\vert \lambda_m\right\vert}{\left\vert\mu_m\right\vert}r,\quad \forall~ k\geq 0,
\end{equation*}
which gives \eqref{eq:Z1_k}, as well as \eqref{eq:Z1_tail} thanks to \eqref{eq:M_Y}.
\end{proof}

\subsubsection{Estimates for \boldmath $A\left(Df\left(\bx+y\right)-A^{\dag}\right)z$ \unboldmath}\label{sec:Z^2}

This Sub-subsection is devoted to the exposition of a bound for $A\left(Df\left(\bx+y\right)-A^{\dag}\right)z$, which
constitutes the second (and last) part of a bound for $Z$. This bound is detailed  in Proposition~\ref{prop:Z2}.
\medskip

Recall the assumption that the nonlinear part $N$ is polynomial of degree 2. Hence, $Df\left(\bx+y\right)$ can be written as a finite Taylor expansion
\begin{equation*}
Df\left(\bx+y\right) = Df\left(\bx\right) + D^2f\left(\bx\right)(y),
\end{equation*}
and
\begin{equation} \label{eq:splitting}
\left(Df\left(\bx+y\right)-A^{\dag}\right)z = \left(Df\left(\bx\right)-A^{\dag}\right)z + D^2f\left(\bx\right)(y,z).
\end{equation}
We are going to bound the two terms of \eqref{eq:splitting} separately. Let us denote by $\sigma$ the coefficient of degree 2 of $f$, that is  $D^2f\left(\bx\right)(y,z)=2\sigma(y\ast z)$. We bound this convolution product thanks to the following result:

\begin{lemma}\label{lem:maj}
Let $s\geq 2$ be an algebraic decay rate and $n\geq 6$, let
 $L\geq 1$ be computational parameters. For $x,y\in\Omega^s$ and for any $k \ge 0$,
\begin{equation*}
\left\vert \left(x\ast y\right)_k\right\vert \leq \alpha_k^s(n)\frac{\Vert x\Vert_s \Vert y\Vert_s}{\omega_k^s},
\end{equation*}
where
\begin{equation*}
\alpha_k^s(n) \bydef
\left\lbrace
\begin{aligned}
& 1+2\sum_{l=1}^L\frac{1}{l^s}+\frac{2}{(s-1)L^{s-1}},\quad k=0,\\
& 2+2\sum_{l=1}^L\frac{1}{l^s}+\frac{2}{(s-1)L^{s-1}}+\sum_{l=1}^{k-1}\frac{k^s}{l^s(k-l)^s},\quad 1\leq k<n,\\
& 2+2\sum_{l=1}^L\frac{1}{l^s}+\frac{2}{(s-1)L^{s-1}}+2\left(\frac{n}{n-1}\right)^s + \left(\frac{4\ln(n-2)}{n}+\frac{\pi^2-6}{3}\right)\left(\frac{2}{n}+\frac{1}{2}\right)^s,\quad k\geq n.
\end{aligned}
\right.
\end{equation*} 
\end{lemma}
\begin{proof}
See \cite{MR3077902} for a proof of this bound and \cite{MR3125637} for a similar bound for $1<s<2$.
\end{proof}
\begin{remark}
It is important to notice here that $\alpha_k^s(n)=\alpha_n^s(n)$ for all $k\geq n$. From now on, we assume that $m$ is taken larger or equal to $6$, which will allow us to use Lemma~\ref{lem:maj} with $n=m$. Note that this condition is not stringent,
 since in practice more than 6 modes are usually needed in order to get a good numerical solution $\bx$.
\end{remark}

We begin by bounding the first term of \eqref{eq:splitting}.
\begin{proposition}
Define $C^1= C^1 (\bx) = \left( C^1_k (\bx) \right)_{k \ge 0} $ component-wise by
\begin{equation*}
C^1_0(\bx) \bydef 0,\quad C^1_k(\bx) \bydef 2\left\vert \sigma\right\vert\sum_{l=m-k}^{m-1} \frac{\left\vert \bx_{l}\right\vert}{\omega^s_{k+l}},\ \forall~ 1\leq k\leq m-1,
\end{equation*}
and
\begin{equation*}
C^1_{m+k}(\bx) \bydef \frac{2\left\vert \sigma\right\vert\alpha^s_{m}(m)\left\Vert \bx\right\Vert_s}{\omega^s_{m+k}},\quad \forall~ k\geq 0.
\end{equation*}
Then for all $z\in B_{0}(r)$
\begin{equation*}
\left\vert \left(Df(\bx)-A^{\dag}\right)z\right\vert \leq C^1 (\bx)r.
\end{equation*}
\end{proposition}
\begin{proof}
According to the definition of $A^{\dag}$ in (\ref{eq:opertor_A_dag}), we see that
\begin{align*}
\left(\left(Df(\bx)-A^{\dag}\right)z\right)_F &= \left(Df(\bx)z\right)_F - Df^{(m)}(\bx)z_F - \left( \begin{array}{c} 0\\ 0 \\ . \\ . \\ 0 \\ \beta_{m-1}\, z_m \\ \end{array} \right)\\
&= 2\sigma\left((\bx\ast z)_F-(\bx\ast z_F)_F\right),
\end{align*}
where in the convolution product, $z_F$ must be understood as the infinite vector $(z_F,0,\ldots,0,\ldots)^T$. Therefore, 
$\left(\left(Df(\bx)-A^{\dag}\right)z\right)_0 = 0$, and for all $z\in B_{0}(r)$,
\begin{equation*}
\left\vert\left(Df(\bx)-A^{\dag}\right)z\right\vert_k \leq  2\left\vert \sigma\right\vert r\sum_{l=m-k}^{m-1} \frac{\left\vert \bx_{l}\right\vert}{\omega^s_{k+l}},\quad \forall~~ 1\leq k\leq m-1.
\end{equation*}
Then, remembering that $Df(\bx)=\cL + DN(\bx)$ and (\ref{eq:opertor_A_dag}), we see that
\begin{equation*}
\left(\left(Df(\bx)-A^{\dag}\right)z\right)_I = \left(DN(\bx)z\right)_I = 2\sigma\left(\bx\ast z\right)_I,
\end{equation*}
so that using Lemma~\ref{lem:maj}, for all $z\in B_{0}(r)$, we end up with the bound
\begin{equation*}
\left\vert\left(Df(\bx)-A^{\dag}\right)z\right\vert_{m+k} \leq \frac{2\left\vert \sigma\right\vert\alpha^s_{m+k}(m)\left\Vert \bx\right\Vert_s}{\omega^s_{m+k}} r, \quad \forall~ k\geq 0 . \qedhere
\end{equation*}
\end{proof}

We now bound the second term of \eqref{eq:splitting}.
\begin{proposition}
Recall \eqref{eq:W^s} and define $C^2=\left( C^2_k\right)_{k\geq 0}$ component-wise by
\begin{equation*}
C^2_k \bydef \frac{2\left\vert \sigma\right\vert\alpha^s_k(m)}{\omega_k^s},\quad \forall~k\geq 0.
\end{equation*}
Then for all $y,z\in B_{0}(r)$
\begin{equation*}
\left\vert D^2f\left(\bx\right)(y,z)\right\vert \leq C^2 r^2.
\end{equation*}
\end{proposition}
\begin{proof}
Remembering that $D^2f\left(\bx\right)(y,z)=2\sigma(y\ast z)$, this is
 a consequence of Lemma~\ref{lem:maj}.
\end{proof}
 
Finally,
\begin{equation*}
\left\vert A\left(Df\left(\bx+y\right)-A^{\dag}\right)z\right\vert \leq \left\vert A\right\vert \left( C^1(\bx)r + C^2 r^2\right),
\end{equation*}
and we are left to bound $\left\vert A\right\vert C^1(\bx)$ and $\left\vert A\right\vert C^2$.
\begin{proposition}\label{prop:D1}
Let $M$ be an integer satisfying \eqref{eq:M_A} and \eqref{eq:M_Y}. We define
 $D^1=\left( D^1_k \right)_{k \ge 0}$ component-wise by
\begin{equation}
\label{eq:D1_F}
D^1_F(\bx)  \bydef  \left\vert A_m\right\vert C^1_F(\bx) + \frac{2\left\vert \beta_{m-1}\right\vert\eta \left\vert \sigma\right\vert\alpha^s_{m}(m) \Vert \bx\Vert_s}{C_1(1-\theta)\omega_m^{s+s_L}} \left\vert A_m\right\vert_{c_{m-1}},
\end{equation}
\begin{align}
\label{eq:D1_k}
D^1_{m+k}(\bx) &\bydef  \left(\left\vert A_m\right\vert_{r_{m-1}} C^1_F(\bx) + \frac{2\left\vert \beta_{m-1}\right\vert\left\vert A_m\right\vert_{m-1,m-1}\eta \left\vert \sigma\right\vert\alpha^s_{m}(m) \Vert \bx\Vert_s}{C_1(1-\theta)\omega_m^{s+s_L}} \right)\eta\frac{\left\vert \lambda_m\right\vert}{\left\vert\mu_m\right\vert}\theta^k \nonumber\\
&+ \frac{2\eta \left\vert \sigma\right\vert\alpha^s_{m}(m) \Vert \bx\Vert_s}{C_1\omega_{m+k}^{s+s_L}} \left(\sum_{l=0}^{k}\theta^{k-l}\left(\frac{m+k}{m+l}\right)^{s+s_L}  + \frac{\theta}{1-\theta}\right),\quad \forall~ 0\leq k< M,
\end{align}
\begin{align}
\label{eq:D1_M}
D^1_{m+M}(\bx) &\bydef  \left(\left\vert A_m\right\vert_{r_{m-1}} C^1_F(\bx) + \frac{2\left\vert \beta_{m-1}\right\vert\left\vert A_m\right\vert_{m-1,m-1}\eta \left\vert \sigma\right\vert\alpha^s_{m}(m) \Vert \bx\Vert_s}{C_1(1-\theta)\omega_{m}^{s+s_L}} \right)\eta\frac{\left\vert \lambda_m\right\vert}{\left\vert\mu_m\right\vert}\theta^M \nonumber\\
&+ \frac{2\eta \left\vert \sigma\right\vert\alpha^s_{m}(m) \Vert \bx\Vert_s}{C_1\omega_{m+M}^{s+s_L}} \left(\chi  + \frac{\theta}{1-\theta}\right),
\end{align}
and 
\begin{equation}
\label{eq:D1_tail}
D^1_{m+k}(\bx) \bydef  D^1_{m+M}(\bx)\frac{\omega_{m+M}^s}{\omega_{m+k}^s},\quad \forall~ k>M.
\end{equation}
Then 
\begin{equation*}
\left\vert A\right\vert C^1(\bx) \leq D^1(\bx).
\end{equation*}
\end{proposition}
\begin{proof}
Thanks to (\ref{eq:operator_A}), 
\begin{equation*}
\left(\left\vert A\right\vert C^1(\bx)\right)_F \leq \left\vert A_m\right\vert C^1_F(\bx) + \left\vert \beta_{m-1}\right\vert \left\vert U_I^{-1}L_I^{-1} C^1_I(\bx)\right\vert_0 \left\vert A_m\right\vert_{c_{m-1}},
\end{equation*}
and using (\ref{eq:maj_unif_expl})
\begin{equation*}
\left\vert U_I^{-1}L_I^{-1} C^1_I(\bx)\right\vert_0 \leq \frac{\eta \Vert C^1_I(\bx) \Vert_s}{C_1(1-\theta)\omega_m^{s+s_L}} \leq \frac{2\eta \left\vert \sigma\right\vert\alpha^s_{m}(m) \Vert \bx\Vert_s}{C_1(1-\theta)\omega_m^{s+s_L}},
\end{equation*}
so that \eqref{eq:D1_F} holds.
Still thanks to (\ref{eq:operator_A}),
\begin{align*}
\left(\left\vert A\right\vert C^1(\bx)\right)_I &\leq \left\vert \lambda_m\right\vert\left\vert A_m\right\vert_{r_{m-1}} C^1_F(\bx) \left\vert w_I\right\vert + \left\vert U_I^{-1}L_I^{-1} C^1_I(\bx)\right\vert + \left\vert \lambda_m\right\vert \left\vert \beta_{m-1}\right\vert\left\vert A_m\right\vert_{m-1,m-1} \left\vert U_I^{-1}L_I^{-1} C^1_I(\bx)\right\vert_0 \left\vert w_I\right\vert \\
&\leq \left\vert \lambda_m\right\vert\left(\left\vert A_m\right\vert_{r_{m-1}} C^1_F(\bx) + \frac{2\left\vert \beta_{m-1}\right\vert\left\vert A_m\right\vert_{m-1,m-1}\eta \left\vert \sigma\right\vert\alpha^s_{m}(m) \Vert \bx\Vert_s}{C_1(1-\theta)\omega_m^{s+s_L}} \right)\left\vert w_I\right\vert + \left\vert U_I^{-1}L_I^{-1} C^1_I(\bx)\right\vert.
\end{align*}
Using (\ref{eq:maj_w_I}) and (\ref{eq:maj_unif_expl}),
 we get 
\begin{align*}
\left(\left\vert A\right\vert C^1(\bx)\right)_{m+k} &\leq \left(\left\vert A_m\right\vert_{r_{m-1}} C^1_F(\bx) + \frac{2\left\vert \beta_{m-1}\right\vert\left\vert A_m\right\vert_{m-1,m-1}\eta \left\vert \sigma\right\vert\alpha^s_{m}(m) \Vert \bx\Vert_s}{C_1(1-\theta)\omega_m^{s+s_L}} \right)\eta\frac{\left\vert \lambda_m\right\vert}{\left\vert\mu_m\right\vert}\theta^k \nonumber\\
&+ \frac{2\eta \left\vert \sigma\right\vert\alpha^s_{m}(m) \Vert \bx\Vert_s}{C_1\omega_{m+k}^{s+s_L}} \left(\sum_{l=0}^{k}\theta^{k-l}\left(\frac{m+k}{m+l}\right)^{s+s_L}  + \frac{\theta}{1-\theta}\right),\quad \forall~ 0\leq k< M,
\end{align*}
so that \eqref{eq:D1_k} holds, and using (\ref{eq:maj_w_I}) and (\ref{eq:maj_unif_tail}), we get
\begin{align*}
\left(\left\vert A\right\vert C^1(\bx)\right)_{m+M} &\leq \left(\left\vert A_m\right\vert_{r_{m-1}} C^1_F(\bx) + \frac{2\left\vert \beta_{m-1}\right\vert\left\vert A_m\right\vert_{m-1,m-1}\eta \left\vert \sigma\right\vert\alpha^s_{m}(m) \Vert \bx\Vert_s}{C_1(1-\theta)\omega_{m}^{s+s_L}} \right)\eta\frac{\left\vert \lambda_m\right\vert}{\left\vert\mu_m\right\vert}\theta^M \nonumber\\
&+ \frac{2\eta \left\vert \sigma\right\vert\alpha^s_{m}(m) \Vert \bx\Vert_s}{C_1\omega_{m+M}^{s+s_L}} \left(\chi  + \frac{\theta}{1-\theta}\right),
\end{align*}
so that \eqref{eq:D1_M} holds. As before, \eqref{eq:D1_tail} follows from \eqref{eq:M_Y}.
\end{proof}
We get similar results for the second order term.

\begin{proposition}\label{prop:D2}
Let $M$ be an integer satisfying \eqref{eq:M_A} and \eqref{eq:M_Y}. Define $D^2=\left( D^2_k \right)_{k \ge 0}$ component-wise by
\begin{equation*}
D^2_F \bydef \left\vert A_m\right\vert C^2_F + \frac{2\left\vert \beta_{m-1}\right\vert\eta \left\vert \sigma\right\vert\alpha^s_{m}(m)}{C_1(1-\theta)\omega_{m}^{s+s_L}} \left\vert A_m\right\vert_{c_{m-1}},
\end{equation*}
\begin{align*}
D^2_{m+k}&\bydef \left(\left\vert A_m\right\vert_{r_{m-1}} C^2_F + \frac{2\left\vert \beta_{m-1}\right\vert\left\vert A_m\right\vert_{m-1,m-1}\eta \left\vert \sigma\right\vert\alpha^s_{m}(m) }{C_1(1-\theta)\omega_{m}^{s+s_L}} \right)\eta\frac{\left\vert \lambda_m\right\vert}{\left\vert\mu_m\right\vert}\theta^k \\
&+ \frac{2\eta \left\vert \sigma\right\vert\alpha^s_{m}(m) }{C_1\omega_{m+k}^{s+s_L}} \left(\sum_{l=0}^{k}\theta^{k-l}\left(\frac{m+k}{m+l}\right)^{s+s_L}  + \frac{\theta}{1-\theta}\right),\quad \forall~ 0\leq k< M,
\end{align*}
\begin{align*}
D^2_{m+M}&\bydef \left(\left\vert A_m\right\vert_{r_{m-1}} C^2_F + \frac{2\left\vert \beta_{m-1}\right\vert\left\vert A_m\right\vert_{m-1,m-1}\eta \left\vert \sigma\right\vert\alpha^s_{m}(m) }{C_1(1-\theta)\omega_{m}^{s+s_L}} \right)\eta\frac{\left\vert \lambda_m\right\vert}{\left\vert\mu_m\right\vert}\theta^M + \frac{2\eta \left\vert \sigma\right\vert\alpha^s_{m}(m) }{C_1\omega_{m+M}^{s+s_L}} \left(\chi  + \frac{\theta}{1-\theta}\right),
\end{align*}
and
\begin{equation*}
D^2_{m+k} \bydef D^2_{m+M}\frac{\omega_{m+M}^s}{\omega_{m+k}^s},\quad \forall~ k>M.
\end{equation*}
Then 
\begin{equation*}
\left\vert A\right\vert C^2 \leq D^2.
\end{equation*}
\end{proposition}

Finally we can sum up all the computations of this Sub-subsection and state the following result:
\begin{proposition}\label{prop:Z2}
Let $M$ be an integer satisfying \eqref{eq:M_A} and \eqref{eq:M_Y}. We define
 $D^1$ (resp. $D^2$) as in Proposition~\ref{prop:D1} (resp. Proposition~\ref{prop:D2}) and  let
\begin{equation*}
Z^2(r) \bydef D^1(\bx)r + D^2r^2.
\end{equation*}
Then for all $y,z \in B_{0}(r)$
\begin{equation*}
A\left(Df\left(\bx+y\right)-A^{\dag}\right)z \leq Z^2(r).
\end{equation*}
\end{proposition}

Putting this together with Proposition~\ref{prop:Z1}, we end up with the following result:
\begin{proposition}\label{prop:Z}
Let $M$ be an integer satisfying \eqref{eq:M_A} and \eqref{eq:M_Y}. Let $$Z(r) \bydef Z^1(r)+Z^2(r).$$ Then for all $y,z \in B_{0}(r)$,
\begin{equation*}
\left\vert DT\left(\bx +y\right)z\right\vert \leq Z(r).
\end{equation*}
\end{proposition}

\subsection{The radii polynomials and interval arithmetics}
\label{sec:radii_pol}

All the work done up to now in Sections \ref{sec:tridiagonal} and \ref{sec:rigorous_computational method} can be summarized in the following statement:

\begin{theorem}
Let $s>1$, and $s_L>0$. Assume that $f$ is a map from $\Omega^s$ to $\Omega^{s-s_L}$ of the form $f=\cL+N$, where $\cL$ is a tridiagonal operator satisfying \eqref{eq:L_k_tridiagonal}, \eqref{eq:asump1_tridiag} and \eqref{eq:asump2_tridiag}, and where the non linear part $N$ is quadratic. Assume that for some $m\geq 6$ we have computed an approximate zero of $f$, of the form $\bx=(\bx_0,\ldots,\bx_{m-1},0,\ldots,0,\ldots)$, and $D$ an approximate inverse of $Df^{(m)}(\bx)$. Consider
\begin{equation*}
T: 
\left\{
\begin{aligned}
&\Omega^s \to \Omega^s, \\
&x \mapsto x-Af(x),
\end{aligned}
\right.
\end{equation*}
where $A$ is defined as in \eqref{eq:operator_A}. Take $M$ satisfying \eqref{eq:M_A} and \eqref{eq:M_Y} and $L\geq 0$ a computational parameter. Then the bound $Y$ defined in Proposition~\ref{prop:Y} satisfies \eqref{eq:Y} and for all $r>0$, the bound $Z(r)$ defined in Proposition~\ref{prop:Z} satisfies \eqref{eq:Z}.
\end{theorem}

Now that we have found bounds $Y$ and $Z(r)$ that satisfy \eqref{eq:Y} and \eqref{eq:Z}, we must find a radius $r>0$ such that $\| Y+Z(r)\|_s<r$ in order to apply Theorem~\ref{thm:Tfixedpt}. By definition of the norm $\left\Vert \cdot\right\Vert_s$,
 it amounts to find an $r>0$ such that, for every $k\geq 0$, the {\em radii polynomial} $P_k(r)$ satisfies
\begin{equation*}
P_k(r) \bydef Y_k + Z_k(r) - \frac{r}{\omega^s_k}<0.
\end{equation*}
Note that since we constructed $Y$ and $Z$ in such a way that for every $k\geq M$,
\begin{equation*}
Y_{m+k}=Y_{m+M}\frac{\omega_{m+M}^s}{\omega_{m+k}^s} \quad\text{and}\quad Z_{m+k}=Z_{m+M}\frac{\omega_{m+M}^s}{\omega_{m+k}^s},
\end{equation*}
it is enough to find an $r>0$ such that for all $0\leq k\leq m+M$, $P_k(r)<0$. In order to do so, we  numerically
compute, for each $0\leq k \leq m+M$,
\begin{equation*}
I_k \bydef \{ r>0 \ |\ P_k(r)<0\},
\end{equation*}
and
\begin{equation*}
I \bydef \bigcap_{k=0}^{m+M} I_k.
\end{equation*}
If $I$ is empty, then the proof fails, and we should try again with some larger parameters $m$ and $M$. If $I$ is non empty, we pick an $r \in I$ and check rigorously, using the interval arithmetics package INTLAB \cite{Ru99a}, that for all $0\leq k\leq m+M$, $P_k(r)<0$, which according to Theorem~\ref{thm:Tfixedpt}, proves that $T$ defined in (\ref{eq:T}) is a contraction on $B_s(\bx,r)$, thus yielding the existence of a unique solution of $f(x)=0$ in $B_s(\bx,r)$.

\section{An example of application} \label{sec:example}

We present in this Section an example of equation, for which it is possible to
 apply the method developed in this paper. We first explain the link between the equation that we study (cf. \eqref{eq:example} below) and the tridiagonal operator defined in Section~\ref{sec:tridiagonal}. Then,
 we explain what are in this example the values of the various constants and parameters of
our  method.
\medskip

Equations of the following form:
\begin{align}\label{eq:example}
& -(2 + \cos \xi) u''(\xi) + u(\xi) = -\sigma u(\xi)^2 + g(\xi), \\ 
& u'(0) = u'(\pi) = 0, \nonumber
\end{align}
where $g$ is a $2\pi$-periodic even smooth function, fall into the framework developed in Section~\ref{sec:tridiagonal}. Consider indeed the cosine Fourier expansions of $u$ and $g$:
\[
u(\xi) =\sum_{k\in\mathbb{Z}} x_k \cos(k \xi),~~~~ g(\xi) = \sum_{k\in\mathbb{Z}} g_k \cos(k \xi).
\]
Then, (\ref{eq:example}) can be rewritten as $f(x)=0$, where
\begin{equation*}
f_0(x) \bydef x_0 + x_1 +\sigma \left(x\ast x\right)_0 - g_0,
\end{equation*}
and for all $ k\geq 1$,
\begin{equation} \label{eq:f_k_example}
f_k(x) \bydef
 \frac{1}{2} (k-1)^2 x_{k-1}+ (1 + 2k^2) x_k + \frac{1}{2}(k+1)^2 x_{k+1} 
+ \sigma \left(x\ast x\right)_k  - g_k.
\end{equation}
We see that the linear part of \eqref{eq:f_k_example} is,
 as in (\ref{eq:L_k_tridiagonal}), given by 
\[
\cL_k(x) =\lambda_{k} x_{k-1} + \mu_k x_k + \beta_k x_{k+1},
\]
with 
\begin{equation*}
\mu_0 \bydef 1,\quad \beta_0 \bydef 1,
\end{equation*}
and for all $k\geq 1$,
\[
\lambda_{k} \bydef  \frac{1}{2} (k-1)^2,~~ \mu_k \bydef (1 + 2k^2)~~ {\rm and } ~~ \beta_k \bydef \frac{1}{2}(k+1)^2.
\]
Let us fix some $m\geq 2$. With 
\begin{equation*}
C_1=2, \quad C_2=3 \quad\text{and}\quad \delta =\frac{1}{4}\frac{(m+1)^2}{m^2+\frac{1}{2}},
\end{equation*}
we get
\begin{equation*}
\forall ~Êk\geq 1,\quad \left| \frac{\lambda_k}{k^{2}}\right|,\left| \frac{\mu_k}{k^{2}}\right|,\left| \frac{\beta_k}{k^{2}}\right|\leq C_2,
\end{equation*}
together with
\begin{equation*}
\forall ~ k\geq m,\quad C_1\leq  \left| \frac{\mu_k}{k^{2}}\right| \quad \text{and}\quad \left| \frac{\lambda_k}{\mu_k} \right|, \left| \frac{\beta_k}{\mu_k} \right| \le \delta.
\end{equation*}

We now focus on the example when 
\begin{equation*}
g(\xi) \bydef \frac{1}{2} + 3\cos(\xi) + \frac{1}{2}\cos(2\xi),
\end{equation*}
so that $u(\xi)=\cos(\xi)$ is a trivial solution for $\sigma=0$. 
We are going to use rigorous computations in order to prove the existence of solutions for $\sigma\neq 0$,
 and to compute these solutions.

%\subsection{Results}

Starting from $\sigma=0$, we first use standard pseudo-arclength continuation techniques to numerically get some nontrivial approximate solutions for $\sigma\neq 0$. We computed 1250 different solutions (675 for $\sigma >0$ and 675 for $\sigma <0$). See Figure~\ref{fig:diag_sigma} for a diagram summing up those computations, where each point represents a solution of \eqref{eq:example}. 

\begin{figure}[htbp]
\begin{center}
\includegraphics[width=8cm]{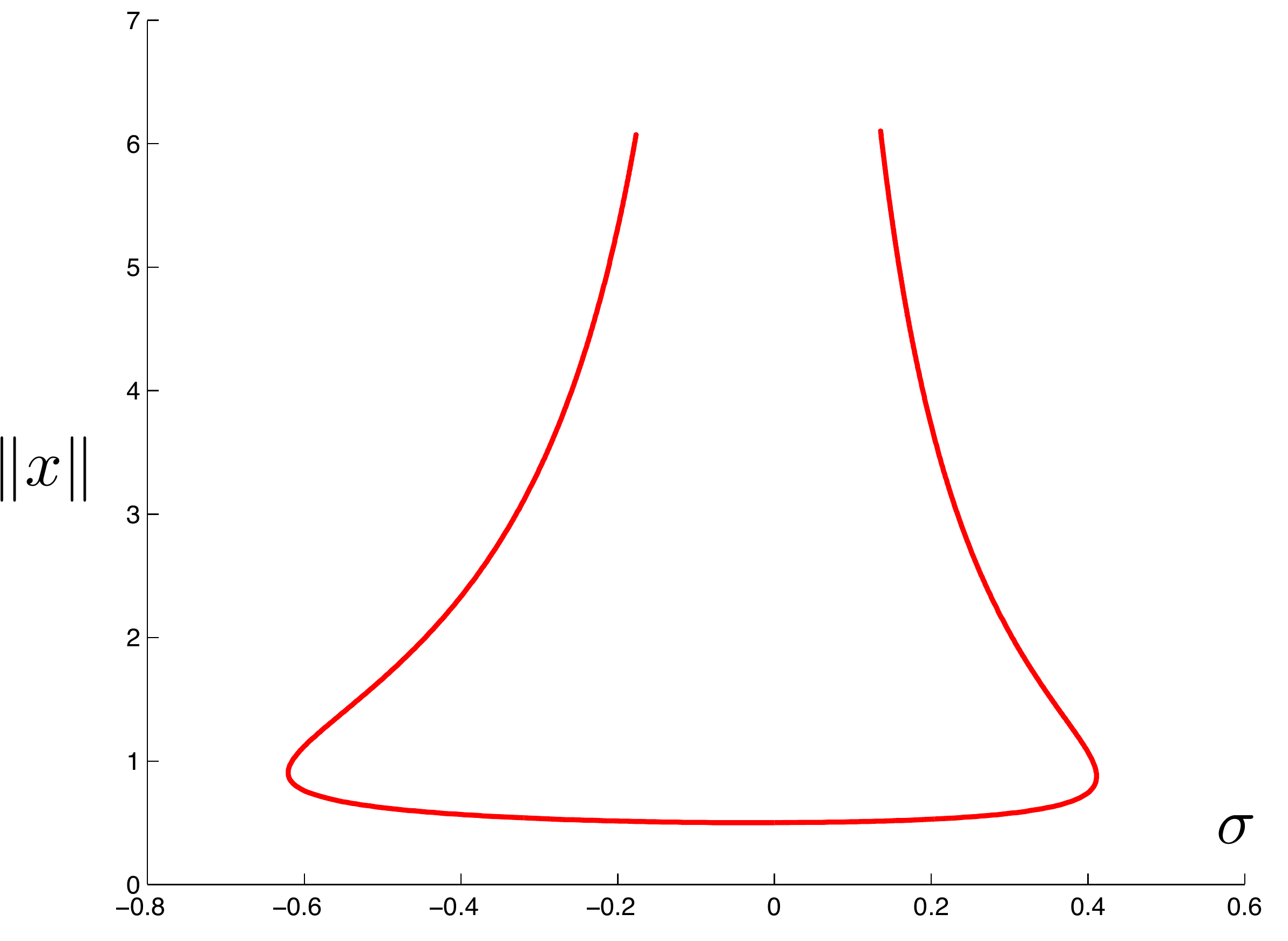}
\caption{Branch of solutions of \eqref{eq:example}.}
\label{fig:diag_sigma}
\end{center}
\end{figure}

Then we use the rigorous computation method described in this paper to prove, for each numerical solution, the existence of a true solution in a small neighbourhood of the numerical approximation. We keep $m=20$ Fourier coefficients
 for the numerical computation, and use $M=20$ and the decay rate $s=2$ for the proof. The bounds of Lemma~\ref{lem:maj} as well as the error on $\tilde\omega$ (\ref{eq:err_tilde_w}) are computed with $L=100$. For each numerical solution, the proof is
successful. The set $I$ defined in Section~\ref{sec:radii_pol} on which all radii polynomials should be negative always contains $[4 \times10^{-11},10^{-4}]$, and we rigorously prove 
using interval arithmetics that they are indeed all negative for $r=10^{-10}$. Hence the assumptions
 of Theorem~\ref{thm:Tfixedpt} hold and as a consequence,  within a ball of radius $r=10^{-10}$ in $\Omega^s$ centered on the numerical approximation, there exists a unique solution to (\ref{eq:example}). Therefore the existence of the solutions represented in
Figure~\ref{fig:diag_sigma} is rigorously proven, within a margin of error that is too small to be depicted. The codes used
to perform the proofs can be found in \cite{webpage}.

Notice that existence of  solutions of (\ref{eq:example}) could certainly have been obtained in different and more classical ways, for example using perturbative methods when $\sigma$ is close to $0$, or using a variational approach (that is, considering (\ref{eq:example}) as the Euler-Lagrange equation related to the critical points of a functional), or even using topological tools such as the Leray-Schauder theory. The  advantage of our method is that it gives us more quantitative information than those approaches: indeed it enables to provide more than one solution for some values of $\sigma$, and, maybe more importantly, it gives a very precise localization of this (or these) solution(s) in terms of Fourier coefficients (something that looks very hard to obtain with qualitative PDEs methods).

\section{Conclusion and Perspectives} \label{sec:conclusion}

A first interesting future direction of research would consist in  adapting
 our approach to the rigorous computation connecting orbits of ODEs (using spectral methods). For instance, we would like to investigate the possibility of combining Hermite spectral methods with our approach to compute homoclinic orbits (e.g. see \cite{MR2333727,MR2709491}). Since the differential operator in frequency space of the Hermite functions is tridiagonal, adapting our method to this class of operator could lead to a new rigorous numerical method for connecting orbits. 

It would also be interesting to adapt our method to the case of solutions belonging to the sequence space 
$$
\ell_\nu^1 = \{ x=(x_k)_{k \ge 0} : \|x\|_\nu \bydef \sum_{k\ge 0} |x_k| \nu^k<\infty \}
$$
for some $\nu \ge 1$. With this choice of Banach space, we could use
 the fact that $\ell_\nu^1$ is naturally a Banach algebra under discrete convolutions. This could greatly simplify the nonlinear analysis. 

Note that assumption \eqref{eq:asump2_tridiag} requires the tridiagonal operator to have symmetric ratios between the diagonal terms and the upper and lower diagonal terms. This is a restriction that could hopefully be relaxed. Since many interesting problems involve tridiagonal operators with non symmetric ratios (as in the case of differentiation in frequency space of the Hermite functions), we believe that this is a promising route to follow.

%, and that a relaxed assumption could be
%%
%\[
%\frac{|\lambda_k|}{|\mu_k|}\leq \delta_1, \frac{|\beta_k|}{|\mu_k|}\leq  
%\delta_2,~~~~  \delta_1 \delta_2 \leq \delta <1|/4, ~~~~\forall  ~k\geq k_0.
%\]
%%

Finally, generalizing our approach to problems with block-tridiagonal structures could also be a valuable project.

\section*{Acknowledgement}
The research leading to this paper was partially funded by the french ``ANR blanche'' project Kibord: ANR-13-BS01-0004.

\end{document}